\documentclass[reqno]{amsart}
\usepackage{amsfonts,graphicx,float,amssymb}
\usepackage {tikz}
\usetikzlibrary{arrows,positioning}
\tikzset{
vertex1/.style ={circle, inner sep=1.5pt, fill=white, draw},
vertex2/.style ={circle, inner sep=1.5pt, fill=black, draw},
Sedge/.append style={dashed, shorten <=6pt, shorten >=6pt},
Redge/.append style={->,>=stealth', shorten <=6pt, shorten >=6pt},
Fedge/.append style={dotted, thick, shorten <=6pt, shorten >=6pt},
Ledge/.append style={->,>=stealth', dashed, shorten <=6pt, shorten >=6pt},
Nedge/.append style={->,>=stealth', shorten <=6pt, shorten >=6pt},
LLedge/.append style={dashed, shorten >=6pt, shorten <=6pt},
NNedge/.append style={shorten <=6pt, shorten >=6pt},
every loop/.append style={min distance=12mm,shorten <=6pt, shorten >=6pt}
}

\makeatletter                                                  %
\def\th@plain{\slshape}                                        %
\makeatother                                                   %

\newcommand{\Zbb}{\mathbb{Z}}
\newcommand{\Qbb}{\mathbb{Q}}
\newcommand{\Rbb}{\mathbb{R}}

\newcommand{\ud}{\,\mathrm{d}}

\newcommand{\pp}{_{>0}}
\newcommand{\m}{^{-1}}

\newcommand{\TTcal}{\mathcal{TT}}
\newcommand{\Gcal}{\mathcal{G}}
\newcommand{\Scal}{\mathcal{S}}
\newcommand{\Acal}{\mathcal{A}}
\newcommand{\Fcal}{\mathcal{F}}
\newcommand{\Lcal}{\mathcal{L}}
\newcommand{\Mcal}{\mathcal{M}}
\newcommand{\Zcal}{\mathcal{Z}}
\newcommand{\abf}{\mathbf{a}}
\newcommand{\bbf}{\mathbf{b}}
\newcommand{\Zbf}{\mathbf{Z}}
\newcommand{\Ybf}{\mathbf{Y}}
\newcommand{\Xbf}{\mathbf{X}}

\mathchardef\mhyphen="2D
\newcommand{\SRScal}{SR\mhyphen\mathcal{S}}
\newcommand{\LNScal}{LN\mhyphen\mathcal{S}}

\newcommand{\newword}[1]{\textsl{#1}}
\newcommand{\vect}[3]{#1_#2,\ldots ,#1_#3}
\newcommand{\angles}[1]{\langle #1 \rangle}

\newcommand{\set}[1]{\{ #1 \}}

\newcommand{\ppmatrix}[4]{\bigl(\begin{smallmatrix}#1&#2\\#3&#4\end{smallmatrix}\bigr)}
\newcommand{\bbmatrix}[4]{\bigl[\begin{smallmatrix}#1&#2\\#3&#4\end{smallmatrix}\bigr]}

\DeclareMathSymbol{\upharpoonright}{\mathrel}{AMSa}{"16}

\DeclareMathOperator{\PGL}{PGL}
\DeclareMathOperator{\PSL}{PSL}
\DeclareMathOperator{\PP}{P}

\theoremstyle{plain}
\newtheorem{theorem}{Theorem}[section]
\newtheorem{lemma}[theorem]{Lemma}
\newtheorem{corollary}[theorem]{Corollary}

\theoremstyle{definition}
\newtheorem{definition}[theorem]{Definition}
\newtheorem{remark}[theorem]{Remark}
\newtheorem{example}[theorem]{Example}

\begin{document}

\bibliographystyle{plain}

\sloppy

\title[The Serret theorem]{Slow continued fractions, transducers,\\
and the Serret theorem}

\author[]{Giovanni Panti}
\address{Department of Mathematics\\
University of Udine\\
via delle Scienze 206\\
33100 Udine, Italy}
\email{giovanni.panti@uniud.it}

\begin{abstract}
A basic result in the elementary theory of continued fractions says that two real numbers share the same tail in their continued fraction expansions iff they belong to the same orbit under the projective action of $\PGL_2\Zbb$. This result was first formulated in
Serret's \emph{Cours d'alg\`ebre sup\'erieure}, so we'll refer to it as to the Serret theorem.

Notwithstanding the abundance of continued fraction algorithms in the literature, a uniform treatment of the Serret result seems missing. In this paper we show that there are finitely many possibilities for the groups $\Sigma\le\PGL_2\Zbb$ generated by the branches of the Gauss maps in a large family of algorithms, and that each $\Sigma$-equivalence class of reals is partitioned in finitely many tail-equivalence classes, whose number we bound. Our approach is through the finite-state transducers that relate Gauss maps to each other. They constitute opfibrations of the Schreier graphs of the groups, and their synchronizability ---which may or may not hold--- assures the a.e.~validity of the Serret theorem.
\end{abstract}


\thanks{\emph{2010 Math.~Subj.~Class.}: 11A55; 37A45.}

\maketitle

\section{Introduction}\label{ref1}

Let $\alpha$, $\beta$ be irrational numbers, with infinite regular continued fraction expansions $[0,a_1,a_2,\ldots]$, $[0,b_1,b_2,\ldots]$, respectively. It is a classical fact that these expansions have the same tail (i.e., there exists $t_1,t_2\ge0$ such that $a_{t_1+n}=b_{t_2+n}$ for every $n\ge1$) if and only if $\alpha$ and $\beta$ are conjugated by an element of the extended modular group $\PGL_2\Zbb$. This result first appeared as \S16 of the third edition (1866) of
Serret's \emph{Cours d'alg\`ebre sup\'erieure}~\cite{serret66}, 
the second edition (1854) making no mention of continued fractions; easily accessible modern references are~\cite[\S10.11]{hardywri85}, \cite[\S9.6]{leveque77}, \cite[\S2.7]{borwein_et_al14}.
An equivalent reformulation is that $\alpha, \beta$ (without loss of generality in the real unit interval) are in the same $\PGL_2\Zbb$-orbit iff they have the same eventual orbit under the \newword{Gauss map} $T:x\mapsto x\m-\lfloor x\m\rfloor$ (see Figure~\ref{fig1} right). The key point here is that the c.~f.~expansion of, say, $\alpha$ is nothing else than its $T$-symbolic orbit: $T^t(\alpha)\in \bigl[ (a_{t+1}+1)\m,a_{t+1}\m \bigr)\phantom{]}$ for every $t\ge0$.
We refer to \cite[Chapter~7]{CornfeldFomSi82}, \cite[Chapter~3]{einsiedlerward} and references therein for the interpretation of continued fractions in terms of dynamical systems. 

Besides the regular ``Floor'' one, a great number of continued fraction algorithms appear in the literature, the complex of them forming a large passacaglia on the theme of the euclidean algorithm. As a definitely incomplete list we cite the Ceiling, Nearest Integer, Even, Odd, Farey fractions~\cite{baladivallee05}, the $\alpha$-fractions~\cite{nakada81}, \cite{arnouxschmidt13}, and the $(a,b)$-fractions~\cite{katokugarcovici12}, not to mention algorithms with coefficients in rings of algebraic integers and multidimensional continued fractions. Asking for the status of the Serret result for these systems is then quite natural.

In this paper we give a fairly complete answer for the algorithms in a certain specific class, namely the class of accelerations of Gauss-type maps arising from finite unimodular partitions of a unimodular interval. After setting notation and stating a few well known facts, we introduce our class in \S\ref{ref3}; we provide various explicit examples, showing that our class, albeit nonexhaustive, contains many important and much studied algorithms. It is a fortunate fact that the validity ---or lack of it--- of the Serret property is untouched by the acceleration process, so that we can restrict to ``slow'' algorithms. In~\S\ref{ref4} we associate a graph $\Gcal_T$ to each such algorithm $T$, and show that $\Gcal_T$ is an opfibration of the Schreier graph of the group $\Sigma_T$ generated by the branches of $T$. The r\^ole of $\Sigma_T$ is clearly crucial; indeed, if $\alpha$, $\beta$ have the same eventual $T$-orbit then they must necessarily be $\Sigma_T$-equivalent. Thus, the key question becomes ``In how many tail-equivalence classes is partitioned a given $\Sigma_T$-equivalence class?'', the Serret property amounting to the constant answer ``Precisely one''. In~\S\ref{ref12} we show that the index of each $\Sigma_T$ in $\PGL_2\Zbb$ is at most~$8$, so that there are finitely many possibilities for these groups. In~\S\ref{ref15} we introduce finite-state transducers, and employ them in two ways: in Lemma~\ref{ref18} to relate different algorithms to each other, and in Lemma~\ref{ref29} to compute the expansion of a rational function of $\alpha$ directly from the expansion of $\alpha$; neither use is new, see~\cite[\S3.5]{grabinerlagarias01} for the first and~\cite{raney73}, \cite{liardetstambul98} for the second.
In Theorem~\ref{ref20} we answer the question cited above: every $\Sigma_T$-equivalence class is partitioned in finitely many tail-equivalence classes, whose number is bounded by the \newword{defect} of the algorithm. We also give an explicit criterion, Corollary~\ref{ref27}, for deciding the validity of the Serret property for a given $T$. In the final~\S\ref{ref5} we relate the synchronizability of the graph~$\Gcal_T$ to the almost-everywhere (w.r.t.~the Lebesgue measure) validity of the Serret theorem.

Before defining our class, we fix notation and recall a few well-known facts.
We denote the group $\PGL_2\Zbb$ and its index-$2$ subgroup $\PSL_2\Zbb$ by $\Pi$ and $\Gamma$, respectively. We set names for certain elements of $\Pi$, using square brackets to emphasize that matrices are taken up to multiplication by $\pm1$:
\[
L=\begin{bmatrix}
1&\\
1&1
\end{bmatrix},\;
N=\begin{bmatrix}
1&1\\
&1
\end{bmatrix},\;
S=\begin{bmatrix}
&-1\\
1&
\end{bmatrix},\;
R=\begin{bmatrix}
1&-1\\
1&
\end{bmatrix},\;
F=\begin{bmatrix}
&1\\
1&
\end{bmatrix}.
\]
The following facts are well known:
\begin{enumerate}
\item $\Gamma=\angles{L,N}=\angles{S,R}$; it has the presentation $\angles{s,r|s^2=r^3=1}$ and hence is isomorphic to $Z_2 * Z_3$.
\item The set of elements of $\Gamma$ having nonnegative entries is the free monoid $\Mcal$ on the two free generators $L,N$.
\item $\Pi=\angles{S,R,F}$ with the presentation $\angles{s,r,f|s^2=r^3=f^2=1, fs=sf, fr=r^2f}$; it is isomorphic to $\angles{S,F}*_{\angles{F}}\angles{R,F}\simeq D_2*_{Z_2}D_3$.
\item The automorphism group of $\Gamma$ is $\Pi$, acting by conjugation. The outer automorphism group of $\Pi$ has order $2$, and is generated by the involution
\[
\alpha:\begin{cases}
F \leftrightarrow F,\\
S \leftrightarrow \bbmatrix{-1}{}{}{1},\\
R \leftrightarrow R^2.
\end{cases}
\]
Note that $\Gamma$ is not $\alpha$-invariant. Moreover,
the usual distinction of elements of $\Gamma$ in elliptic, parabolic, and hyperbolic (according to the absolute value of the trace being less than, equal, or greater than $2$) is destroyed: for example, $\alpha$ exchanges the parabolic $N^2$ with the hyperbolic $\bbmatrix{1}{1}{1}{2}$.
\item $\Pi$ acts on $\PP^1\Rbb$ in the standard projective way: if $A=\bbmatrix{a}{b}{c}{d}$, then $A*\alpha=(a\alpha+b)/(c\alpha+d)$. If $\Sigma$ is a subgroup of $\Pi$ and $\alpha,\beta$ are in the same $\Sigma$-orbit, then we say that $\alpha$ and $\beta$ are \newword{$\Sigma$-equivalent}.
\end{enumerate}

\section{Gauss-type maps and the groups they generate}\label{ref3}

We identify (actually, we define) a continued fraction algorithm with the corresponding Gauss-type map; in our setting, the latter are defined as follows. Let us say that an interval with rational endpoints $[p/q,p'/q']$ is \newword{unimodular} if $\det\ppmatrix{p}{p'}{q}{q'}=-1$. A \newword{unimodular partition} of the base unimodular interval $\Delta=[0,\infty]=[0/1,1/0]$ is a finite family $\set{\vect\Delta0{{n-1}}}$ (of cardinality at least $2$) of unimodular intervals such that $\bigcup_a\Delta_a=\Delta$ and distinct $\Delta_a$'s intersect at most in a common endpoint; we always assume that $\vect\Delta0{{n-1}}$ are listed in consecutive order, with $0\in\Delta_0$ and $\infty\in\Delta_{n-1}$.
For every index $a$, we choose arbitrarily $e_a\in\set{+1,-1}$.

\begin{definition}
The \newword{slow continued fraction algorithm} determined by the family of pairs $\set{(\Delta_a,e_a)}$ as above is the map $T:\Delta\to\Delta$ which is induced on $\Delta_a=[p/q,p'/q']$ by the matrix
\[
A_a^{-1}:=
\begin{bmatrix}
&1\\
1&
\end{bmatrix}^{(e_a+1)/2}
\begin{bmatrix}
p&p'\\
q&q'
\end{bmatrix}^{-1}.
\]
If $\Delta_a$ and $\Delta_b$ are consecutive and $e_a=e_b$, then the above definition is ambiguous in their common vertex $p'/q'$. In this case we consider $T$ as a multivalued map, admitting both $0$ and $\infty$ as $T$-images of $p'/q'$; this ambiguity may occur at most once along the $T$-orbit of a point, necessarily rational.
\end{definition}

We usually specify $T$ by providing the finite set $\set{\vect A0{{n-1}}}$ of matrices in~$\Pi$ with nonnegative entries whose inverses determine $T$. Note that the defining intervals $\Delta_a$ are recovered by $\Delta_a=A_a*\Delta$. We identify matrices with the maps they induce, and we say that the $A_a$'s are the \newword{inverse branches} of $T$. 

At least one of the extreme points $0,\infty$ is inevitably an indifferent fixed point either for $T$ or for $T^2$ (\newword{indifferent} means that the one-sided derivative at the point is $+1$). One removes such points by accelerating the algorithm, as follows.

\begin{definition}
Let\label{ref25} $T$ be as above, let $0\le i\le j\le n-1$, and let $E=E(i,j)=\bigcup\set{\Delta_a:i\le a \le j}$; we allow the possibility of removing one or both endpoints from the interval $E$.
The \newword{first-return map}, or \newword{accelerated continued fraction algorithm}, is the map $T_E:E\to E$ defined as follows. Given $x\in E$, let $r(x)=\min\set{t\ge1:T^t(x)\in E}\in[1,2,3,\ldots,+\infty]$, and set
\[
T_E(x)=\begin{cases}
T^{r(x)}(x), & \text{if $r(x)\not=+\infty$};\\
\text{undefined}, & \text{otherwise}.
\end{cases}
\]
\end{definition}

\begin{example}
The\label{ref7} set $\set{LF,N}$ (equivalently, the set $\set{([0,1],-1),
([1,\infty],+1)}$)
determines the map whose graph is in Figure~\ref{fig1} left. Since we want $\infty$ to appear as an ordinary point, we draw graphs by conjugating $E$ to $[0,1]$ via an appropriate projective transformation, in this case $L:E=\Delta\to [0,1]$.
The fixed point $\infty$ is indifferent, and by inducing on $E=[0,1)$ we get the usual Gauss map of Figure~\ref{fig1} right, whose inverse branches are $\set{LFN^a:a\ge0}=\bigl\{\bbmatrix{}{1}{1}{a}:a\ge1\bigr\}$.
\begin{figure}[h!]
\includegraphics[width=5cm]{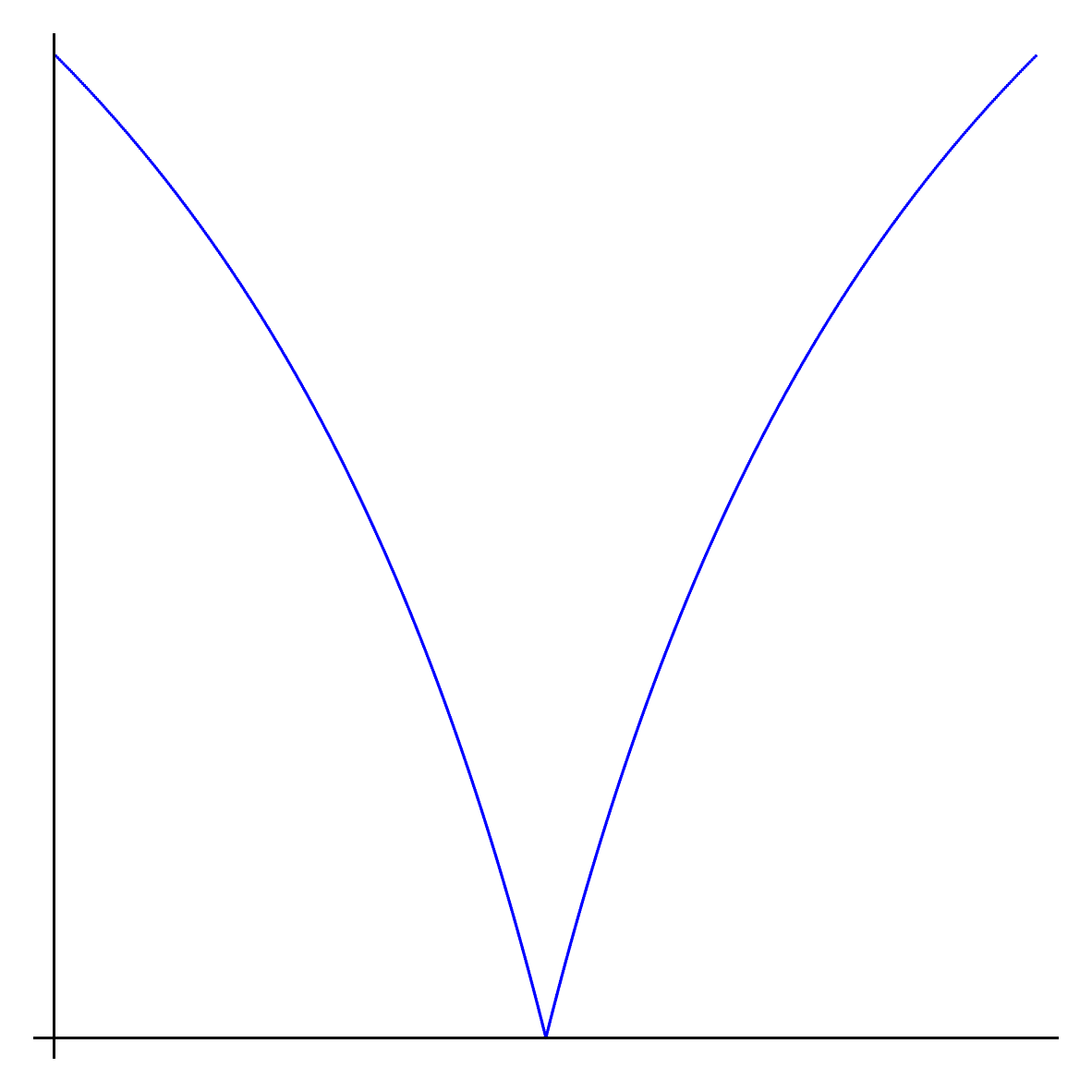}\qquad
\includegraphics[width=5cm]{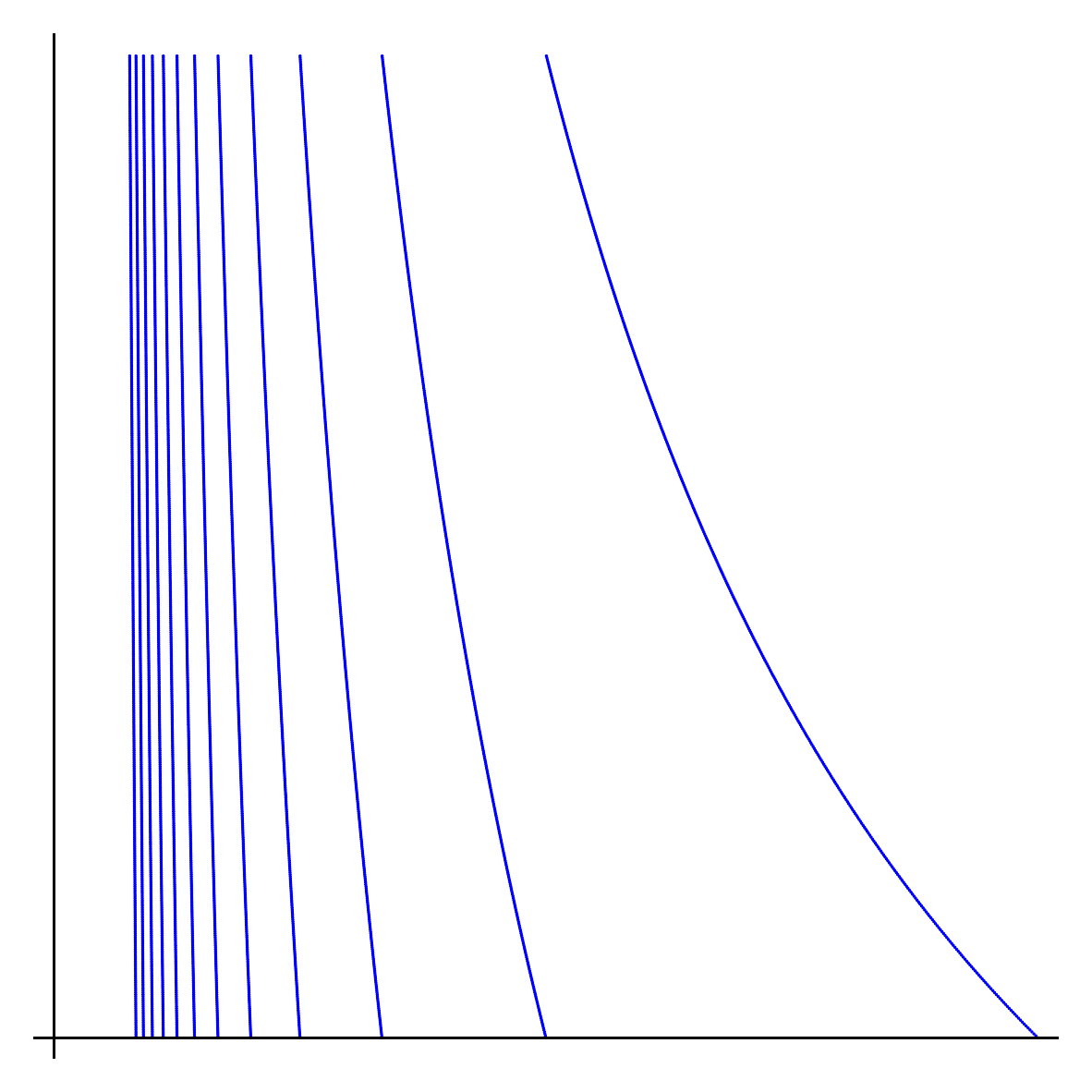}
\caption{}
\label{fig1}
\end{figure}
\end{example}

\begin{example}
Conjugating\label{ref6} the slow algorithm of Example~\ref{ref7} by $F$ we get
the Farey map $\set{L,NF}$ on $[0,\infty]$; see, e.g., \cite{heersink16} and references therein.
Inducing on $(1,\infty]$ we get a conjugated Gauss map with inverse branches $\bigl\{\bbmatrix{a}{1}{1}{}:a\ge1\bigr\}$, whose finite products give the classical matrices $\bbmatrix{p_n}{p_{n-1}}{q_n}{q_{n-1}}$.
\begin{figure}[h!]
\includegraphics[width=5cm]{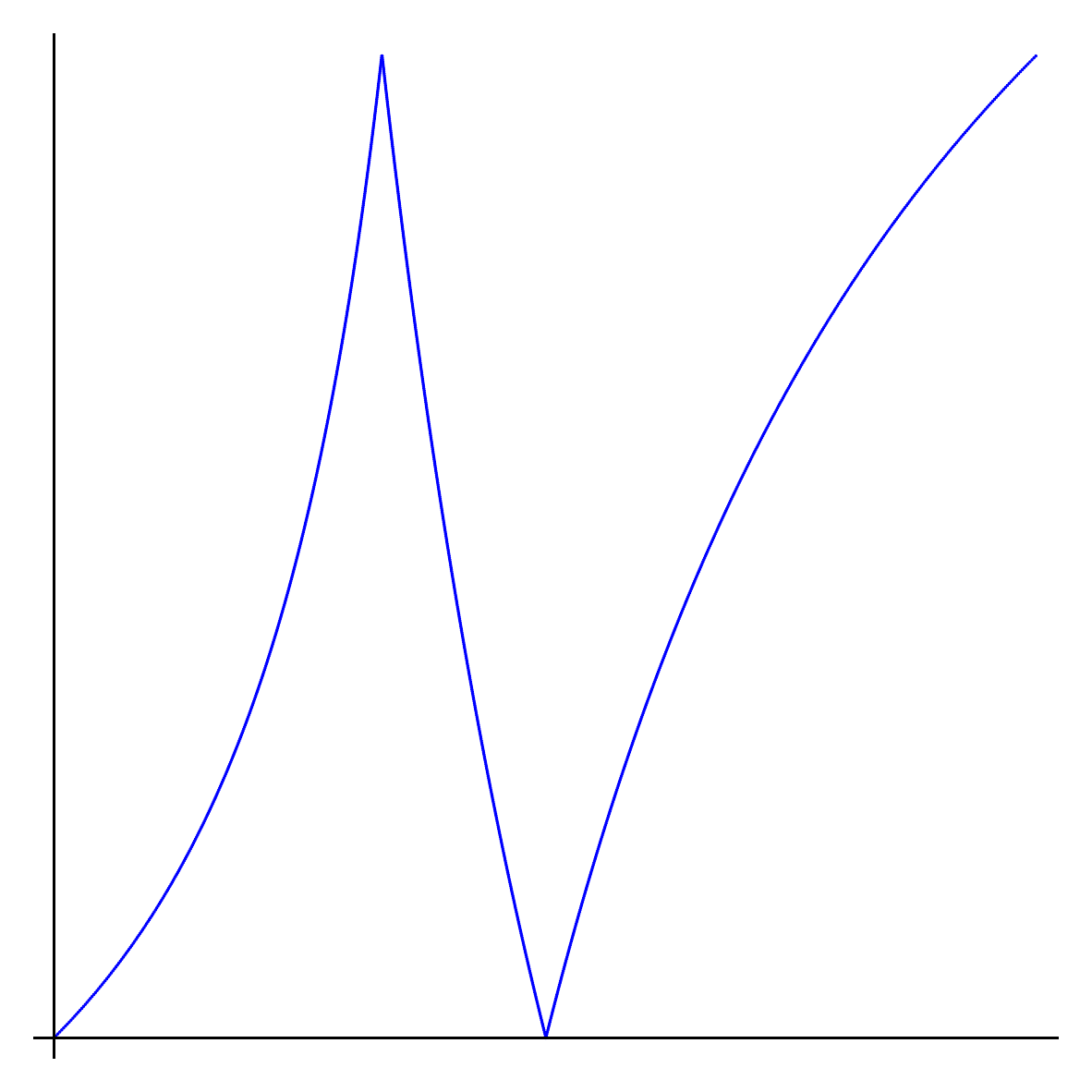}\qquad
\includegraphics[width=5cm]{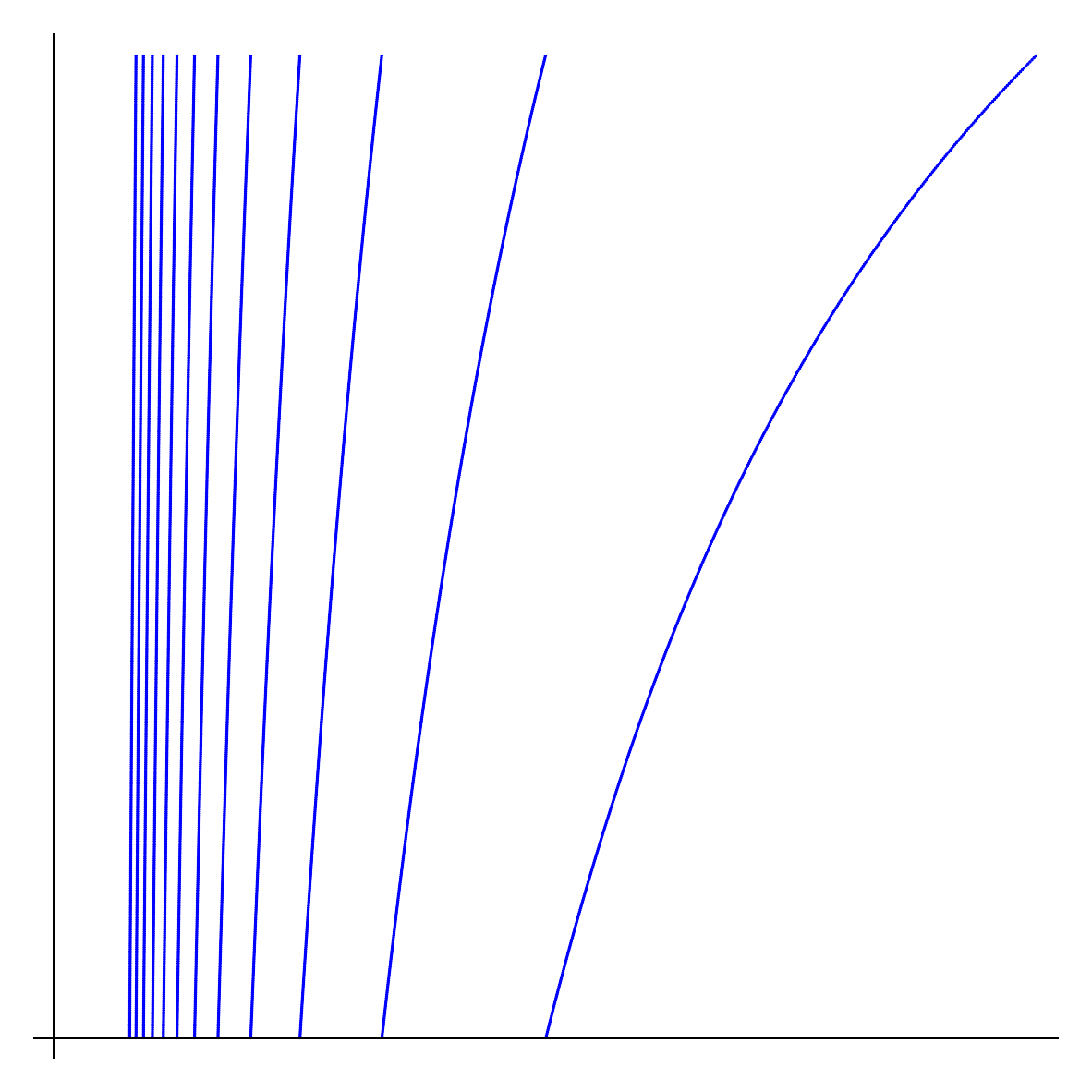}
\caption{}
\label{fig2}
\end{figure}
\end{example}

\begin{example}
Starting\label{ref16} from $\set{L,N}$ and inducing on $[1,\infty]$ we get Zagier's ceiling algorithm (Figure~\ref{fig2} right).
Inducing on $[0,1]$ we get R\'enyi's algorithm; the two are of course conjugated by $F$. See~\cite{zagier75}, \cite{pinner01}, \cite{iommi10} and references therein.
\end{example}

\begin{example}
The~\label{ref24} map $\set{LL,LNF,N}$ in Figure~\ref{fig2} left enumerates all pythagorean triples~\cite{romik08}.
\end{example}

\begin{example}
The maps $\set{LL,LNF,N}$ and $\set{LL,LN,NF}$ are, up to conjugation  by $L$, the even and odd continued fractions in~\cite{bocalinden17}.
\end{example}

If $T$ is a slow algorithm and $T_E$ one of its accelerations, then we say that $T_E$ is a
\newword{c.~f.~algorithm}, or \newword{Gauss-type map}.
Writing $E=E(i,j)$ as in Definition~\ref{ref25}, it is easy to see that the inverse branches of $T_E$ are the matrices of the form $A_aB$ with $i\le a\le j$ and $B$ in the monoid generated by $\set{A_b:b<i\text{ or }j<b}$.
This set of inverse branches is countable (unless $E=\Delta$), and it is clear that the group it generates equals the subgroup $\Sigma_T$ of $\Pi$ generated by $\vect A0{{n-1}}$.
If $\alpha,\beta\in E$ are such that $T_E^{t_1}(\alpha)=T_E^{t_2}(\beta)$ for some
$t_1,t_2\ge0$, then obviously $\alpha,\beta$ are $\Sigma_T$-equivalent; this paper deals with the reverse implication.

\begin{definition}
We\label{ref19} say that \newword{$T_E$ has the tail property}, or that \newword{the Serret theorem holds for $T_E$}, if for every $\alpha,\beta\in E\setminus\Qbb$ that are $\Sigma_T$-equivalent, and whose forward $T_E$-orbit is never undefined, there exist $t_1,t_2\ge0$ such that $T_E^{t_1}(\alpha)=T_E^{t_2}(\beta)$.
\end{definition}

We discuss the above definition in the following remarks.

\begin{remark}
For every slow algorithm $T$ and every rational $\alpha=p/q\in\Delta$, the $T$-orbit of $\alpha$ ends up either in the fixed point $0$, or in the fixed point $\infty$, or in the 2-cycle $\set{0,\infty}$. This is readily proved by observing that whenever $p/q$ is in the topological interior of one of the intervals $\Delta_a$, then $T(\alpha)=p'/q'$ satisfies $0<p'+q'<p+q$. As a consequence, the $T_E$-orbit of any rational number either ends up in $\set{0,\infty}$, or is eventually undefined (this surely happens if $0,\infty\notin E$).

On the other hand, it is easy to construct accelerated maps that are undefined in points not $\Pi$-equivalent. For example, define $T$ by the inverse branches $\set{L,NL,N^2F}$; it is increasing on $\Delta_0=[0,1]$, increasing on $\Delta_1=[1,2]$ with fixed point $(1+\sqrt{5})/2$, and decreasing on $\Delta_2=[2,\infty]$ with fixed point $1+\sqrt{2}$. If we induce on $E=[0,1]$, then the points $\alpha=L*\bigl((1+\sqrt{5})/2\bigr),\, \beta=L*(1+\sqrt{2})\in E$
lie in different quadratic fields, and hence are not $\Pi$-equivalent.
However, $T_E$ is undefined in both $\alpha$ and $\beta$, since their $T$-images are $T$-fixed points external to~$E$.
It is therefore safe to discard rational and eventually undefined points from consideration.
\end{remark}

\begin{remark}
Let $T_{[0,1)}$ denote the classical Gauss map; its inverse branches 
generate~$\Pi$. As noted in~\S\ref{ref1}, the irrational numbers $\alpha,\beta$ have the same eventual $T_{[0,1)}$-orbit (namely, they satisfy the condition in Definition~\ref{ref19}) iff their regular c.~f.~expansions have the same tail. Thus the formulation in Definition~\ref{ref19} is equivalent to the classical one.
\end{remark}

\begin{remark}
If\label{ref8} $T$ has the tail property, then so does any of its accelerations.
Conversely, if $T_E$ has the tail property and $\alpha,\beta\in\Delta\setminus\Qbb$ are $\Sigma_T$-equivalent and such that their $T$-orbits enter $E$ infinitely often, then $T^{t_1}(\alpha)=T^{t_1}(\beta)$ for some $t_1,t_2$.
\end{remark}

\begin{example}
Here\label{ref26} is a simple example of an algorithm for which the tail property fails; we'll construct a more elaborate one in Example~\ref{ref21}. 
\begin{figure}[h]
\includegraphics[width=5cm]{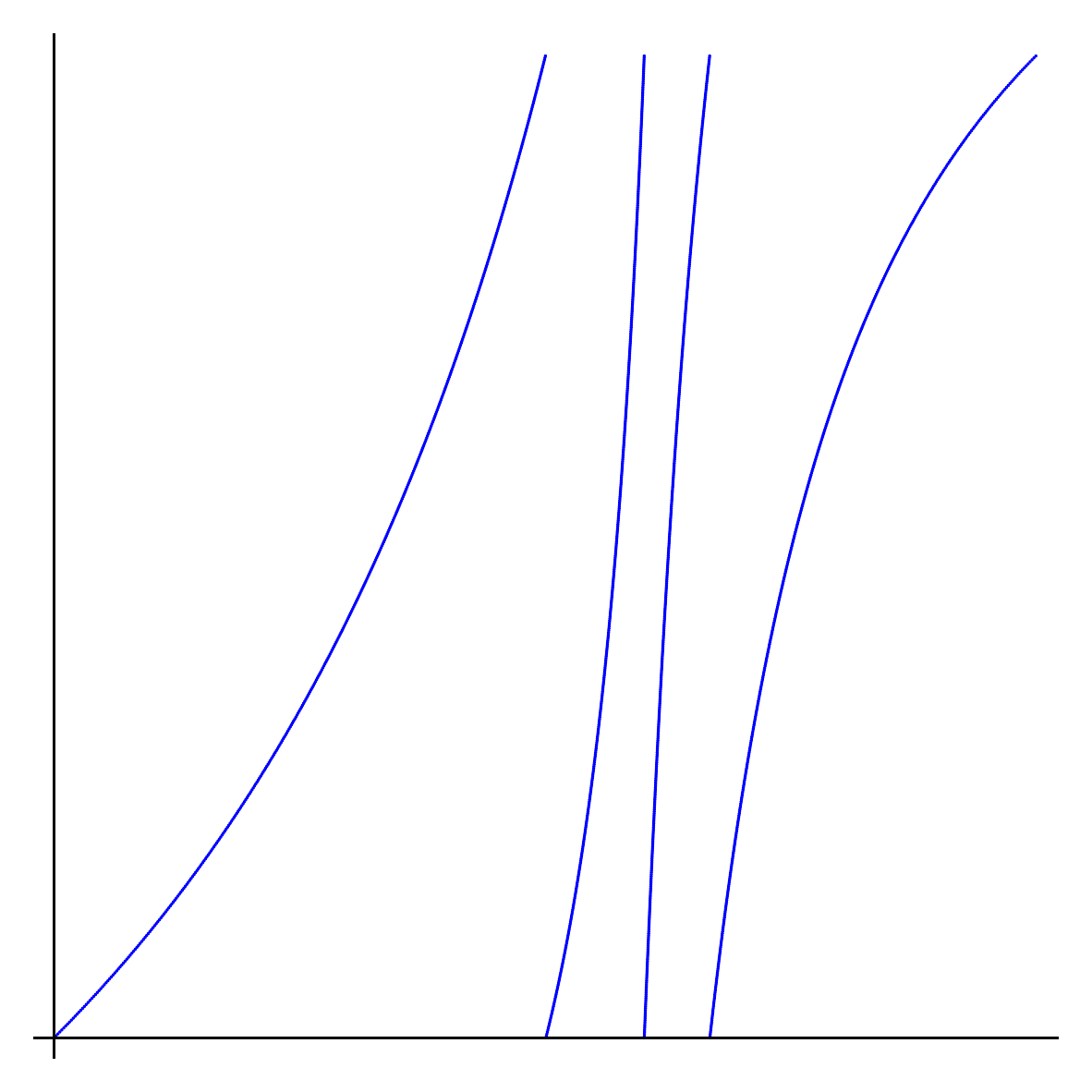}\qquad
\includegraphics[width=5cm]{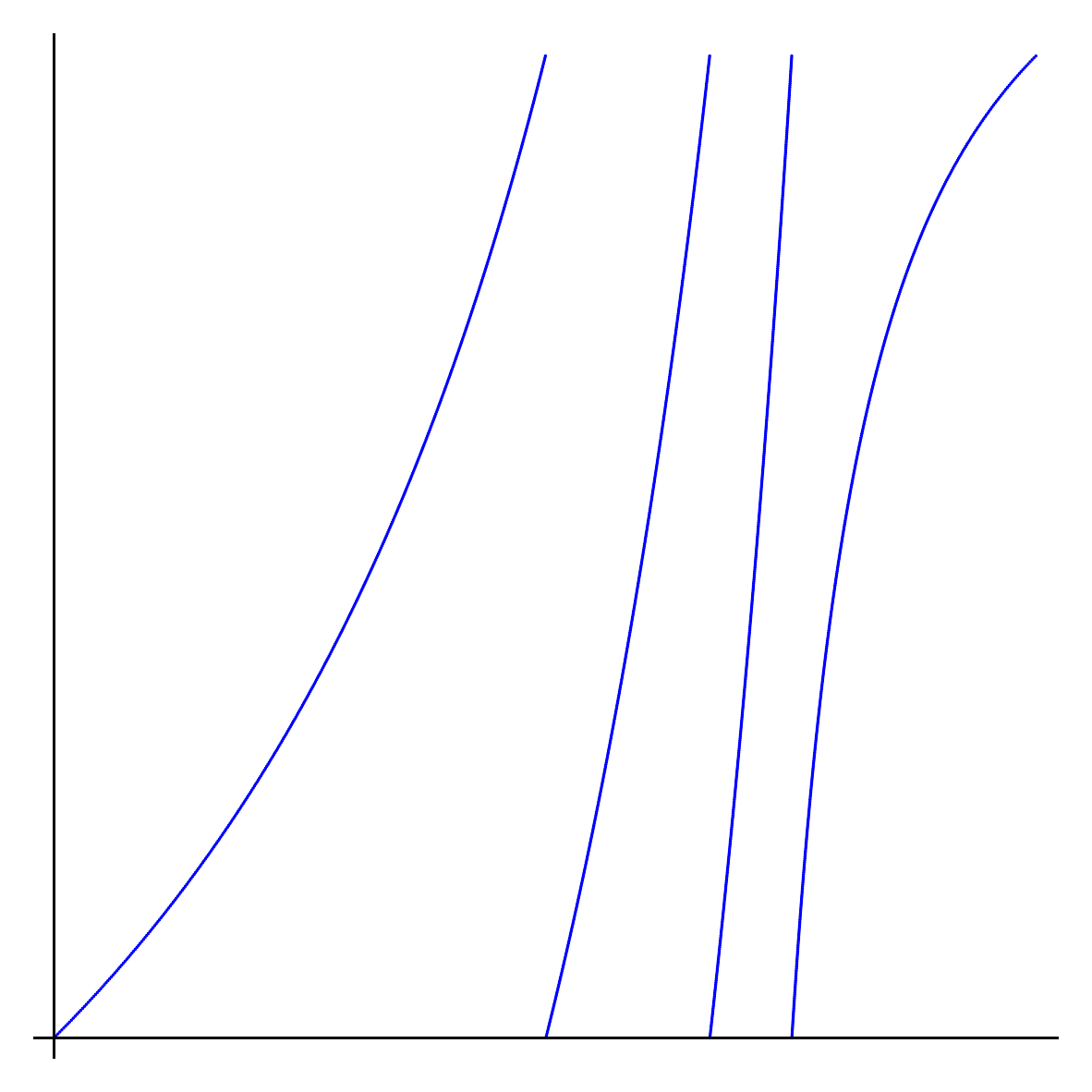}
\caption{}
\label{fig13}
\end{figure}
Let $T$ have inverse branches $\set{L,NLL,NLN,NN}$, and let $T'$ have $\set{L,NL,NNL,NNN}$. Their graphs, shown in Figure~\ref{fig13}, are similar, and obviously $\Sigma_T=\Sigma_{T'}=\Gamma$ (see also Corollary~\ref{ref28}).
We will show in Corollary~\ref{ref27} that $T'$ has the tail property. On the other hand, consider the third inverse branch $A_2=NLN$ of $T$, and let $\alpha=A_2*\alpha=\sqrt{3}$ be the fixed point of $T$ in $\Delta_2$. Of course $\alpha$ has $T$-symbolic orbit $\overline{2}$. However $NA_2N\m=A_3A_0$, and thus $\beta=N*\alpha=\sqrt{3}+1$ has $T$-symbolic orbit $\overline{30}$; hence the tail property fails.
\end{example}

\section{The graph of an algorithm}\label{ref4}

As noted in Remark~\ref{ref8}, the validity of the Serret theorem for a slow algorithm is equivalent to its validity for any acceleration; accordingly, for the rest of this paper we will only consider slow algorithms.
In this section we show that there are finitely many possibilities for the groups $\Sigma_T$; as a matter of fact, all such groups have index at most $8$ in $\Pi$.

We recall that the \newword{Schreier graph} $\Scal=\Scal(\Sigma,\Pi,\set{S,R,F})$ of $\Sigma\le\Pi$ w.r.t.~the generating set $\set{S,R,F}$ is defined as follows:
\begin{itemize}
\item the vertices of $\Scal$ are the right cosets $\Sigma B\in\Sigma\backslash\Pi$, and there is a distinguished vertex $1$ (called the \newword{root}), corresponding to $\Sigma$;
\item each edge is directed and labelled by one of $S,R,F$;
\item there is an $S$-edge from $\Sigma B$ to $\Sigma C$ iff $\Sigma C=\Sigma BS$, and analogously for $R$- and $F$-edges.
\end{itemize}
We'll also need the Schreier graph $\Scal(\Sigma,\Pi,\set{L,N,F})$
taken w.r.t.~the generating set $\set{L,N,F}$. We 
drop reference to $\Pi$ and to the generating set whenever possible; if the generating set is relevant we'll write
$\SRScal(\Sigma)$ and $\LNScal(\Sigma)$.

Schreier graphs are instances of rooted directed connected edge-labelled graphs. A homomorphism $\varphi:\Gcal'\to\Gcal$ of such graphs is a function that maps vertices to vertices, edges to edges, and preserves all the structure (root, labelling, edge source and target).
The homomorphism $\varphi$ is a \newword{covering} if it is surjective on vertices and edges, and \newword{locally trivial}: for each vertex $y$ of $\Gcal'$, $\varphi$ gives a bijection from the set of edges leaving and entering $y$ to
the set of edges leaving and entering $\varphi(y)$. If in the above definition we drop the words ``and entering'' we get the definition of \newword{opfibration} of graphs (the name originates from category theory~\cite[Definition~2.2]{boldivigna02}, an alternative name being \newword{right-covering}~\cite[Definition~8.2.1]{LindMar95}).

If $\Sigma'\le\Sigma$ then $\Sigma'B\mapsto\Sigma B$ gives a covering from $\Scal(\Sigma')$ to $\Scal(\Sigma)$; in particular, the \newword{Cayley graph} $\Scal(1)$ covers any Schreier graph for $\Pi$. The simple form of the relations involving $S,R,F$ makes the drawing of
$\SRScal(1)$ easy: we represent the ``upper part'' of it ---namely, the Cayley graph of $\Gamma$--- in Figure~\ref{fig5}.
The pattern in Figure~\ref{fig5} extends to infinity in all directions,
and $R$-edges are represented as plain arrows. As $S^2=1$,
$S$-edges appear in pairs going in opposite directions; we represent such a pair by a single unoriented dashed edge. We'll discuss later the vertices marked by a black circle. We obtain the full Cayley graph of $\Pi$ by attaching a twin copy of Figure~\ref{fig5} ``under the page''. Each upper vertex is connected to its lower twin by a pair of $F$-edges going in opposite directions, and again represented by a single dotted unoriented edge; $S$-edges are faithfully copied from the upper level to the lower one. $R$-edges are copied as well but, due to the relation $FRF=R\m$, the clockwise $3$-cycles become counterclockwise ones.
Drawing Schreier graphs w.r.t.~the generating set $\set{L,N,F}$ is messier, because the arrows tend to intersect; in Example~\ref{ref11} and later on
we'll use dashed arrows for $L$-edges and plain arrows for $N$-edges, while $F$-edges remain dotted and unoriented. One easily commutes between the two generating sets via the relations $L=R^2S$, $N=RS$, $S=NL\m N$, $R=LN\m$.
\begin{figure}[h!]
\begin{tikzpicture}[x={(1cm,0cm)},y={(-0.5cm,0.866cm)},
scale=0.8]
\node at (0,0) []  {$1$};
\node at (2,1) [vertex1]  {};
\node at (3,2) [vertex2]  {};
\node at (5,3) [vertex1]  {};
\node at (6,3) [vertex2]  {};
\node at (4,4) [vertex1]  {};
\node at (4,5) [vertex2]  {};
\node at (1,-1) [vertex1]  {};
\node at (1,-2) [vertex2]  {};
\node at (2,-3) [vertex1]  {};
\node at (3,-3) [vertex2]  {};
\node at (0,-4) [vertex1]  {};
\node at (-1,-5) [vertex2]  {};
\node at (-1,0) [vertex1]  {};
\node at (-2,1) [vertex1]  {};
\node at (-3,-1) [vertex1]  {};
\node at (-4,-2) [vertex1]  {};
\node at (-5,-4) [vertex1]  {};
\node at (-5,-5) [vertex1]  {};
\node at (-6,-3) [vertex1]  {};
\node at (-7,-3) [vertex1]  {};
\node at (-2,2) [vertex1]  {};
\node at (-1,4) [vertex1]  {};
\node at (0,5) [vertex1]  {};
\node at (-3,3) [vertex1]  {};
\node at (-4,3) [vertex1]  {};
\path (0,0) edge[Redge] (2,1)
            edge[Sedge] (-1,0);
\path (2,1) edge[Redge] (1,-1)
            edge[Sedge] (3,2);
\path (1,-1) edge[Redge] (0,0)
            edge[Sedge] (1,-2);
\path (4,4) edge[Redge] (5,3)
            edge[Sedge] (4,5);
\path (5,3) edge[Redge] (3,2)
            edge[Sedge] (6,3);
\path (-3,-1) edge[Redge] (-2,1)
            edge[Sedge] (-4,-2);
\path (-2,1) edge[Redge] (-1,0)
            edge[Sedge] (-2,2);
\path (2,-3) edge[Redge] (0,-4)
            edge[Sedge] (3,-3);
\path (0,-4) edge[Redge] (1,-2)
            edge[Sedge] (-1,-5);
\path (-1,4) edge[Redge] (-2,2)
            edge[Sedge] (0,5);
\path (-3,3) edge[Redge] (-1,4)
            edge[Sedge] (-4,3);
\path (-5,-4) edge[Redge] (-6,-3)
            edge[Sedge] (-5,-5);
\path (-6,-3) edge[Redge] (-4,-2)
            edge[Sedge] (-7,-3);
\path (3,2) edge[Redge] (4,4);
\path (1,-2) edge[Redge] (2,-3);
\path (-4,-2) edge[Redge] (-5,-4);
\path (-2,2) edge[Redge] (-3,3);
\path (-1,0) edge[Redge] (-3,-1);
\end{tikzpicture}
\caption{}
\label{fig5}
\end{figure}

The free monoid $\Mcal$ mentioned in~\S\ref{ref1}(2) determines a proper full subgraph of the $LN$-Cayley graph of $\Gamma$, and this subgraph is an infinite binary tree; the root $1$ and the black circles in Figure~\ref{fig5} are precisely the vertices of this tree. The $F$-twins ``under the page'' of these vertices correspond to the matrices in $\Pi$ with nonnegative entries and determinant~$-1$. They form a twin binary tree, whose $L,N$ labelling is flipped from the upper one, due to the relation $FLF=N$.
The complex of these two trees, together with the vertical $F$-edges connecting each vertex to its twin, constitutes a subgraph $\TTcal$ of the $LN$-Cayley graph of $\Pi$.

Before returning to our slow algorithms and stating Definition~\ref{ref17} we need an observation about unimodular partitions. Again by~\S\ref{ref1}(2), each unimodular partition of $\Delta=[0,\infty]$ can be obtained (not uniquely) from the trivial partition $\set{\Delta}$ in finitely many steps, each step consisting in choosing an interval $\Delta'=B*\Delta$ of the current partition ($B$ being some element of $\Mcal$), and replacing it with its two \newword{Farey splittings} $\Delta''=BL*\Delta$ and $\Delta'''=BN*\Delta$. Let now $T$ be a slow algorithm; recall that we specify it by providing a finite set $\set{\vect A0{{n-1}}}$ of matrices in $\Pi$ with  nonnegative entries,
which we number according to the left-to-right enumeration of the intervals $\Delta_a=A_a*[0,\infty]$.
Such a set is uniquely determined by $T$, and every $A_a$ is uniquely factorizable as $A_a=B_aF^{e(a)}$, with $B_a\in\Mcal$ and $e(a)\in\set{0,1}$. If we describe the unimodular partition associated to $T$ in terms of successive Farey splittings as above, then the set of $B$'s encountered during the process is precisely the set of left factors in $\Mcal$ of $\vect B0{{n-1}}$. In particular, given any $B$ in this set of left factors, precisely one of the following holds:
\begin{itemize}
\item[(a)] neither $BL$ nor $BN$ are in the set (this holds iff $B$ is one of $\vect B0{{n-1}}$);
\item[(b)] both $BL$ and $BN$ are in the set.
\end{itemize}

\begin{definition}
Let\label{ref17} $T$, $A_a=B_aF^{e(a)}$ be as above. We define the \newword{graph of $T$}, denoted by $\Gcal_T$, as follows. We start from the double binary tree $\TTcal$ described in the penultimate paragraph.
We delete from $\TTcal$ all vertices (and all edges incident to them) except those of the form $B$ and $BF$, where $B$ is a left factor in $\Mcal$ of one of $\vect B0{{n-1}}$. This leaves us with two copies of a finite binary tree, connected by vertical $F$-edges. 
By the previous observation each vertex either has an $L$-child and an $N$-child, which are distinct, or is a leaf and has no children; moreover, the set of leaves is precisely $\set{B_a:a\in\set{0,\ldots,n-1}}\cup
\set{B_aF:a\in\set{0,\ldots,n-1}}$.
Now, for each $a$, if $e(a)=0$ then we glue $B_a$ with the root $1$, and
$B_aF$ with the vertex $F$-connected to $1$, which we'll always denote $2$. Conversely, if $e(a)=1$ we glue $B_a$ with
$2$ and $B_aF$ with the root. This leaves us with a connected rooted graph such that precisely one $L$-, $N$-, and $F$-edge stems from every vertex, while the number of $L$- and $N$-edges entering in a given vertex may be $0$ or greater then~$1$.
\end{definition}

\begin{theorem}
There\label{ref9} exists an opfibration from $\Gcal_T$ to the Schreier graph $\LNScal(\Sigma_T)$; in particular, since $\Gcal_T$ is finite, $\Sigma_T$ has finite index in $\Pi$.
\end{theorem}

\begin{example}
Let\label{ref11} $T$ be determined by $A_0=L^3$, $A_1=L^2NF$, $A_2=LNL$, $A_3=LN^2F$, $A_4=N$. Labelling the vertices for clarity, and omitting the vertical $F$-edges, we obtain the graph~$\Gcal_T$ shown in Figure~\ref{fig7}.
\begin{figure}[h!]
\begin{tikzpicture}[scale=0.7]
\node (1) at (0,10)  []  {$1$};
\node (LN) at (5,9) []  {$LN$};
\node (L)  at (2,8) []  {$L$};
\node (LL)  at (4,6) []  {$LL$};
\node (F)  at (0,4) []  {$2$};
\node (LNF)  at (5,3) []  {$LNF$};
\node (LF)  at (2,2) []  {$LF$};
\node (LLF)  at (4,0) []  {$LLF$};
\path
(1) edge[Ledge] (L)
(1) edge[Nedge,out=110,in=170,loop] (1)
(LN) edge[Ledge] (1)
(LN) edge[Nedge,out=-80,in=5,looseness=1.3] (F)
(L) edge[Ledge] (LL)
(L) edge[Nedge] (LN)
(LL) edge[Ledge,bend left] (1)
(LL) edge[Nedge] (F)
(F) edge[Ledge,out=110,in=170,loop] (F)
(F) edge[Nedge] (LF)
(LF) edge[Ledge] (LNF)
(LF) edge[Nedge] (LLF)
(LNF) edge[Ledge,out=45,in=20,looseness=1.7] (1)
(LNF) edge[Nedge] (F)
(LLF) edge[Ledge,out=180,in=-135,looseness=1.3] (1)
(LLF) edge[Nedge,bend left] (F);
\end{tikzpicture}
\caption{}
\label{fig7}
\end{figure}
\end{example}

\begin{example}
The\label{ref10} minimum number of inverse branches for $T$ is $2$, and in this case there are $4$ possibilities for $\Gcal_T$, shown in Figure~\ref{fig6} (if two distinct vertices are connected by a pair of arrows with the same labelling and going in opposite directions, we draw a single shaft without arrowheads).
\begin{figure}[h!]
\begin{tikzpicture}[scale=0.8]
\node at (0,3) []  {$1$};
\node at (0,1) []  {$2$};
\node at (0,0) []  {$\set{L,N}$};
\path
(0,3) edge[Fedge] (0,1)
(0,3) edge[Nedge,out=30,in=-30,loop] (0,3)
(0,3) edge[Ledge,out=150,in=-150,loop] (0,3)
(0,1) edge[Nedge,out=30,in=-30,loop] (0,3)
(0,1) edge[Ledge,out=150,in=-150,loop] (0,1);
\node at (4,3) []  {$1$};
\node at (4,1) []  {$2$};
\node at (4,0) []  {$\{L,NF\}$};
\path
(4,3) edge[Fedge] (4,1)
(4,3) edge[Ledge,out=150,in=-150,loop] (4,3)
(4,3) edge[Nedge,bend left] (4,1)
(4,1) edge[Ledge,bend left] (4,3)
(4,1) edge[Nedge,out=30,in=-30,loop] (4,1);
\node at (8,3) []  {$1$};
\node at (8,1) []  {$2$};
\node at (8,0) []  {$\{LF,NF\}$};
\path
(8,3) edge[Fedge] (8,1)
(8,3) edge[NNedge,bend left] (8,1)
(8,3) edge[LLedge,bend right] (8,1);
\node at (12,3) []  {$1$};
\node at (12,1) []  {$2$};
\node at (12,0) []  {$\{LF,N\}$};
\path
(12,3) edge[Fedge] (12,1)
(12,3) edge[Ledge,bend right] (12,1)
(12,3) edge[Nedge,out=30,in=-30,loop] (12,3)
(12,1) edge[Nedge,bend right] (12,3)
(12,1) edge[Ledge,out=150,in=-150,loop] (12,1);
\end{tikzpicture}
\caption{}
\label{fig6}
\end{figure}
In the first and third case we directly obtain the Schreier graph for $\Sigma_T$, that turns out to be $\Gamma$ (this is Example~\ref{ref16}) and $\alpha\Gamma$, respectively. In both cases flipping the two vertices is a graph automorphism, corresponding to the fact that $T$ is invariant under conjugation by~$F$. 
In the second and fourth case we have a proper opfibration of the trivial Schreier graph, so $\Sigma_T=\Pi$. Flipping the vertices exchanges the second graph with the fourth and indeed, as remarked in Example~\ref{ref6}, the map in Example~\ref{ref7} is $F$-conjugated to the Farey map.
\end{example}

\begin{proof}[Proof of Theorem~\ref{ref9}]
By construction, all closed circuits in $\Gcal_T$ starting and ending at the root determine a product of $L^{\pm1},N^{\pm1},F$ that belongs to $\Sigma_T$. If no vertex is the target of two distinct edges with the same labelling, then the fact that every vertex is the source of precisely one edge for each edge label implies that it is also the target of precisely one edge for each label. Therefore
$\Gcal_T$ is already a Schreier graph, necessarily ---by the previous remark--- of $\Sigma_T$. If the vertex $v$ is the target of two distinct edges with the same labelling, one originating from $u$ and the other from $w$, then we glue $u$ and $w$ and take the quotient graph. The new root-based closed circuits originating from this process are still in $\Sigma_T$, because the gluing is just a consequence of the cancellation property in groups. We repeat the process, which must eventually terminate, leaving us with the $LN$-Schreier graph of $\Sigma_T$.
\end{proof}

Continuing with Example~\ref{ref11}, the vertices $LN,LL,LNF,LLF$ of Figure~\ref{fig7} must all be glued together, because any of them is the source of an $L$-edge to the root (as well as of an $N$-edge to $F$). The resulting vertex is now the target of an $N$-edge from~$L$, and another one from $LF$. This forces the gluing of $L$ and $LF$, and we are left with the Schreier graph in Figure~\ref{fig8}. Since the root is the unique vertex carrying an $N$-loop, $\Scal(\Sigma_T)$ has a trivial automorphism group; hence $\Sigma_T$ is an
index-$4$ subgroup of $\Pi$ that equals its own normalizer.
\begin{figure}[h!]
\begin{tikzpicture}[scale=0.8]
\node (rt) at (0,2)  []  {$1$};
\node (2) at (0,0) []  {$2$};
\node (3)  at (2,1) [vertex1]  {};
\node (4)  at (4,1) [vertex1]  {};
\path
(rt) edge[Nedge,out=150,in=-150,loop] (rt)
(rt) edge[Ledge] (3)
(rt) edge[Fedge] (2)
(2) edge[Ledge,out=150,in=-150,loop] (2)
(2) edge[Nedge] (3)
(3) edge[Fedge,out=-120,in=-60,loop] (3)
(3) edge[Ledge,bend left] (4)
(3) edge[Nedge,bend right] (4)
(4) edge[Fedge,out=-120,in=-60,loop] (4)
(4) edge[Ledge,bend right] (rt)
(4) edge[Nedge,bend left] (2);
\end{tikzpicture}
\caption{}
\label{fig8}
\end{figure}

\section{Index at most $8$}\label{ref12}

Since $\Pi$ is finitely generated, the following theorem implies that the list of possible $\Sigma_T$'s is finite.

\begin{theorem}
Let\label{ref13} $T$ be a slow continued fraction algorithm. Then the group $\Sigma_T$ generated by the inverse branches of $T$ has index at most $8$ in $\Pi$. All indices from~$1$ to $8$ are realized by some $\Sigma_T$, possibly in non-isomorphic ways.
\end{theorem}
\begin{proof}
The proof is based on two observations and a lemma.
\begin{itemize}
\item[(A)] Every $\Sigma_T$ must contain at least one of $SRS,SRSF,SR^2SF$. Indeed, the observation preceding Definition~\ref{ref17} implies that at least one pair of consecutive intervals
$\Delta_a,\Delta_{a+1}$ in the unimodular partition associated to $T$ arises from the Farey splitting of the interval $\Delta_a\cup\Delta_{a+1}=B*\Delta$, with $B\in\Mcal$ such that $A_a=BLF^{e(a)}$ and $A_{a+1}=BNF^{e(a+1)}$. Therefore $\Sigma_T\ni A_{a+1}\m A_a=F^{e(a+1)}N\m LF^{e(a)}=F^{e(a+1)}SRSF^{e(a)}$. If $e(a),e(a+1)$ have the same parity, then $SRS$ belongs to~$\Sigma_T$, while if they have different parity either $SRSF$ or its $F$-conjugate $SR^2SF$ belongs to $\Sigma_T$.
\item[(B)] The Schreier graph of $\Sigma_T$ cannot contain an infinite path that runs along the positively oriented $L$- and $N$-edges and avoids forever both $1$ and $2$ (which may coincide).
Indeed, if such a path existed then it would be liftable to $\Gcal_T$ by the opfibration property established in Theorem~\ref{ref9} (as a matter of fact, the liftability of positively oriented paths characterizes opfibrations~\cite[\S2.2]{boldivigna02}). But, by the very construction of $\Gcal_T$, such an infinite path cannot exist.
\end{itemize}

Conjugating by $F$ we assume without loss of generality that our fixed $\Sigma=\Sigma_T$ contains at least one of $SRS$, $SRSF$. We will establish Theorem~\ref{ref13} by showing that $\Scal(\Sigma)$ cannot have $9$ or more vertices. We work with $SR$-graphs, and we repeat our drawing conventions: $S$-, $R$-, and $F$-edges are, respectively, dashed, plain, and dotted. Each pair of $S$- or $F$-edges connecting two distinct vertices and going in opposite directions are drawn as a single edge with no arrowheads.

We remark that whenever $(u,v,w)$ is an $R$-cycle in $\Scal(\Sigma)$ ($R$-cycles are always assumed to be nontrivial, otherwise we speak of $R$-loops), then $u$ is $F$-connected to $v$ iff $w$ carries an $F$-loop. This fact follows from the identity $F=RFR$ and will be used several times.

\begin{lemma}
Suppose\label{ref14} we have constructed a subgraph $\Gcal$ of $\Scal(\Sigma)$.
Let $1$, $2$, $\ldots$, $n-1$, $n$ be the vertices of $\Gcal$ (still keeping our standard notation of $1,2$ for the root and the vertex $F$-connected to it). Assume that $n-1$ and $n$ are $S$-connected (possibly $n-1=n$, i.e., there is an $S$-loop at $n-1$), and that $\Gcal$ is \newword{almost complete}, meaning that the mere addition of an $R$-loop at $n$ gives a graph $\Gcal^*$ which is the Schreier graph of some subgroup of~$\Pi$. Then $\Scal(\Sigma)$ is indeed $\Gcal^*$.
\end{lemma}
\begin{proof}
Suppose not. Then $\Scal(\Sigma)$ must contain vertices $n+1,n+2,n+3,n+4$ (and possibly others), all of them not in $\Gcal$, with $(n,n+1,n+2)$ forming an $R$-cycle, $n+1$ $S$-connected to $n+3$, and $n+2$ $S$-connected to $n+4$ ($n+3$ and $n+4$ are not necessarily distinct from $n+1$ and $n+2$). But then we can find in $\Scal(\Sigma)$ an infinite positively oriented $LN$-path touching only vertices not in $\Gcal$. Indeed, we start from either $n+3$ or $n+4$ and move along $L=R^2S$ and $N=RS$ edges arbitrarily. Since~$\Gcal$ is almost complete, there is no risk of touching a vertex in $\Gcal$, unless we land at $n+1$ or at $n+2$. If we land at $n+1$ we move to $n+4$ via $N$, and if we land at $n+2$ we move to $n+3$ via $L$, and continue our errand. The existence of such a path contradicts~(B).
\end{proof}

We now argue by cases.\nopagebreak[4]
\paragraph{\emph{Case 1: $R\in\Sigma$}}
If $SRS\in\Sigma$ then surely the index of $\Sigma$ is 
$\le4$. Indeed every element of $\Pi$ factors (uniquely) as a word in $S$ and $R$, possibly followed by a single occurrence of $F$.
One easily deduces that each right coset must be of the form $\Sigma B$, with $B\in\set{1,S,F,SF}$.

If $SRS\notin\Sigma$ then $SRSF\in\Sigma$ and $F\notin\Sigma$.
This implies that the Schreier graph of~$\Sigma$ in a neighborhood of the root must be of the form given in Figure~\ref{fig9}.
\begin{figure}[h!]
\begin{tikzpicture}[scale=0.8]
\node (1) at (0,2) []  {1};
\node (2) at (0,0) []  {2};
\node (3) at (2,2) []  {3};
\node (4) at (2,0) []  {4};
\node (5) at (4,1) [] {5};
\path
(1) edge[Redge,out=160,in=-160,loop] (1)
(1) edge[Sedge] (3)
(1) edge[Fedge] (2)
(2) edge[Redge,out=160,in=-160,loop] (2)
(2) edge[Sedge] (4)
(3) edge[Fedge] (4)
(3) edge[Redge,bend left] (4)
(4) edge[Redge] (5)
(5) edge[Redge] (3)
(5) edge[Fedge,out=-110,in=-70,loop] (5);
\end{tikzpicture}
\caption{}
\label{fig9}
\end{figure}
Indeed, the $R$-loop in $2$ appears by $R=FR^2F$. The two distinct vertices $3,4$ appear since~$1$ cannot be $S$-connected with either $1$ or $2$ (otherwise $SRS$ would belong to $\Sigma$). The vertex $3$ must be $R$-connected to $4$ since $SRSF\in\Sigma$, and this creates an $R$-cycle involving a fresh vertex $5$. The vertices $3$ and $4$ must be $F$-connected because~$1$ and~$2$ so are, and $F=SFS$. Finally, the $F$-loop in $5$ appears by the remark preceding Lemma~\ref{ref14}.

\paragraph{\emph{Claim}} The only way of completing the graph in Figure~\ref{fig9} to $\Scal(\Sigma)$ is either by adding an $S$-loop to $5$, or by $S$-connecting $5$ to a new vertex $6$ and adding to $6$ an $R$- and an $F$-loop.
\paragraph{\emph{Proof of Claim}} Adding an $S$-loop at $5$ completes the graph and we are done. If we do not do so, we are forced to
$S$-connect $5$ to a new vertex $6$, that must carry an $F$-loop because so does~$5$. The resulting graph is almost complete, so our statement follows from Lemma~\ref{ref14}.

By the claim, $\Sigma$ must have index $5$ or $6$, and this concludes the analysis of Case~1.

We assume now $R\notin\Sigma$, so $1$ belongs to an $R$-cycle $(1,3,5)$
and we have $4$ more cases: $2=1$, $2=3$, $2=5$, and $2\notin\set{1,3,5}$.

\paragraph{\emph{Case 2: $2=1$, i.e., there is an $F$-loop at $1$}}
This simply cannot happen. Indeed, by (A) and $F\in\Sigma$ we have $SRS\in\Sigma$. Now, $1$ cannot be $S$-connected to any of $1$, $3$, $5$, as in that case starting from $1$ and following the path $SRS$ does not bring us back to $1$. Therefore $1$ must be $S$-connected to a new vertex $4$, carrying both an $F$- and an $R$-loop. The graph $\Gcal$ containing the two $S$-connected vertices $1,4$, an $F$-loop at both, and an $R$-loop at $4$, is almost complete. Therefore, by Lemma~\ref{ref14}, there could not be any $R$-cycle at~$1$.

\paragraph{\emph{Case 3: $2=3$}} 
Then there must be an $F$-loop at $5$. Also, $5$ cannot be $S$-connected to either $1$ or $2$, since otherwise $1$ and $2$ would collapse. The vertices $1$ and $2$ cannot be $S$-connected, otherwise~(A) is violated. Therefore, in a neighborhood of the root the Schreier graph of $\Sigma$ is of the form in Figure~\ref{fig10}; it might be $6=1$,  
which is equivalent to $4=2$.
\begin{figure}[h!]
\begin{tikzpicture}[scale=0.8]
\node (6) at (0,2) []  {$6$};
\node (4) at (0,0) []  {$4$};
\node (1) at (2,2) []  {$1$};
\node (2) at (2,0) []  {$2$};
\node (5) at (4,1) []  {$5$};
\path
(6) edge[Sedge] (1)
(6) edge[Fedge] (4)
(4) edge[Sedge] (2)
(2) edge[Fedge] (1)
(1) edge[Redge,bend left] (2)
(2) edge[Redge] (5)
(5) edge[Redge] (1)
(5) edge[Fedge,out=-110,in=-70,loop] (5);
\end{tikzpicture}
\caption{}\label{fig10}
\end{figure}

\paragraph{\emph{Subcase 3.1: $6=1$}}
Then there is an $S$-loop both at $1$ and at $2$. If there is an $S$-loop at $5$ then the graph is complete, while if $5$ is $S$-connected to a new vertex then the latter must carry an $F$-loop. This leaves us with an almost complete graph and Lemma~\ref{ref14} applies. Hence $\Sigma$ has index $3$ or $4$.

\paragraph{\emph{Subcase 3.2: $6\not=1$, and $6$ carries an $R$-loop}}
Due to $R=FR^2F$, $4$ carries an $R$-loop as well. By the same argument as in Subcase~3.1, $\Sigma$ has index $5$ or $6$.

\paragraph{\emph{Subcase 3.3: $6\not=1$, and $6$ carries an $R$-cycle}}
By~(A), $6$ must be $R$-connected to~$4$, so we are left with the situation in Figure~\ref{fig11}, where possibly $8=5$ and/or $9=7$.
\begin{figure}[h!]
\begin{tikzpicture}[scale=0.8]
\node (6) at (0,2) []  {$6$};
\node (4) at (0,0) []  {$4$};
\node (1) at (2,2) []  {$1$};
\node (2) at (2,0) []  {$2$};
\node (5) at (4,1) []  {$5$};
\node (8) at (6,1) []  {$8$};
\node (7) at (-2,1) [] {$7$};
\node (9) at (-4,1) [] {$9$};
\path
(6) edge[Sedge] (1)
(6) edge[Fedge] (4)
(4) edge[Sedge] (2)
(1) edge[Fedge] (2)
(1) edge[Redge,bend left] (2)
(2) edge[Redge] (5)
(5) edge[Redge] (1)
(5) edge[Fedge,out=-110,in=-70,loop] (5)
(5) edge[Sedge] (8)
(8) edge[Fedge,out=-110,in=-70,loop] (8)
(7) edge[Fedge,out=-110,in=-70,loop] (7)
(9) edge[Fedge,out=-110,in=-70,loop] (9)
(7) edge[Sedge] (9)
(6) edge[Redge,bend right] (4)
(4) edge[Redge] (7)
(7) edge[Redge] (6);
\end{tikzpicture}
\caption{}\label{fig11}
\end{figure}
\begin{itemize}
\item If $8=5$ (i.e., $5$ carries an $S$-loop), then Lemma~\ref{ref14} assures us that $\Sigma$ has index $6$ (if $9=7$) or $7$ (if $9\not=7$). Analogously if $9=7$.
\item If $8\not=5$, $9\not=7$, and either $8$ or $9$ carries an $R$-loop, then the other one must carry an $R$-loop as well, again 
by Lemma~\ref{ref14}, so $\Sigma$ has index $8$.
\end{itemize}
We claim that there are no other possibilities, i.e., that any Schreier graph~$\Scal$ extending Figure~\ref{fig11} (with $8\not=5$ and $9\not=7$)  and with strictly more vertices contradicts~(B).
Indeed, by the discussion above, $\Scal$ should contain two $R$-cycles $(8,10,11)$ and $(9,12,13)$. The presence of $F$-loops at $8$ and $9$ forces $\set{8,10,11}\cap\set{9,12,13}=\emptyset$, again by the remark preceding Lemma~\ref{ref14}.
It follows that the set $D=\set{\text{vertices of $\Scal$}}\setminus\set{1,2(=3),4,5,6,7,8,9}$ contains at least the four distinct vertices $10,11,12,13$, as well as the vertices $10',11',12',13'$ $S$-connected to them. But then we can construct an infinite $LN$-path all contained in $D$, just taking care that whenever we land in, say, $10$, we apply $N$ and go to $11'$ (while $L$ would lead us to $5\notin D$). This contradicts~(B), establishes our claim and concludes the analysis of Case~3.

\paragraph{\emph{Case 4: $2=5$}} 
This is analogous to Case~3, except for the possibility, at the beginning of the discussion, that $1$ and $2(=5)$ be $S$-connected. This is now possible (and yields $SRSF\in\Sigma$), while it was not in Case~2. However, if $1$ and $2$ are $S$-connected it is easy to see, by looking at the $S$-edge leaving $3$ and using Lemma~\ref{ref14}, that $\Sigma$ must then have index $3$ or $4$. The rest of the proof is completely analogous to that in Case~3.

\paragraph{\emph{Case 5: $2\notin\set{1,3,5}$}} 
Again, this cannot happen. Indeed, one easily sees that the hypotheses yield the existence of another $R$-cycle $\set{2,4,6}$ with $F$-edges connecting $3$ with $6$ and $5$ with $4$. As in Case~2, condition~(A)
precludes $1$ to be $S$-connected to any of $1,2,3,4,5,6$, and this
forces two new vertices $7\not=8$, which are connected by $S$- and $F$-edges as in Figure~\ref{fig12}. In order to satisfy~(A), precisely two cases are possible.
\begin{figure}[h!]
\begin{tikzpicture}[scale=0.6]
\node (8) at (0,0) []  {$8$};
\node (2) at (3,0) []  {$2$};
\node (4) at (6,-1) []  {$4$};
\node (6) at (5,1) []  {$6$};
\node (7) at (0,3) []  {$7$};
\node (1) at (3,3) []  {$1$};
\node (5) at (6,2) []  {$5$};
\node (3) at (5,4) []  {$3$};
\path
(7) edge[Fedge] (8)
(1) edge[Fedge] (2)
(5) edge[Fedge] (4)
(3) edge[Fedge] (6)
(1) edge[Sedge] (7)
(8) edge[Sedge] (2)
(1) edge[Redge] (3)
(3) edge[Redge] (5)
(5) edge[Redge] (1)
(2) edge[Redge] (4)
(4) edge[Redge] (6)
(6) edge[Redge] (2);
\end{tikzpicture}
\caption{}\label{fig12}
\end{figure}

\paragraph{\emph{Subcase 5.1: There is an $R$-loop at $7$}}
Since $R=FR^2F$, this forces an $R$-loop at $8$ as well. Then, as in Subcase~3.3, we can construct an infinite $LN$-path that avoids the set $\set{1,2,7,8}$ forever; this subcase is therefore impossible.

\paragraph{\emph{Subcase 5.2: There is an $R$-edge from $7$ to $8$}}
This creates an $R$-cycle $(7,8,9)$, with an $F$-loop at $9$. This $F$-loop precludes the possibility of $S$-connecting $9$ with any of $3,4,5,6$ (because none of them carries an $F$-loop). There are now precisely three mutually exclusive possibilities, namely
\begin{itemize}
\item[(i)] $9$ carries an $S$-loop;
\item[(ii)] $9$ is $S$-connected to a new vertex $10$ that carries both an $F$- and an $R$-loop;
\item[(iii)] $9$ is $S$-connected to a new vertex $10$ that carries an $F$-loop and an $R$-cycle $(10,11,12)$.
\end{itemize}
In each of these cases we can again construct an infinite $LN$-path, avoiding forever $\set{1,2,7,8,9}$ (in case~(i)), or 
$\set{1,2,7,8,9,10}$ (in cases~(ii) and~(iii)). This subcase is then impossible as well, and the analysis of Case~5 is completed.

As all indices from $1$ to $8$ have been realized in several nonisomorphic ways during the previous analysis, Theorem~\ref{ref13} is proved. 
\end{proof}

\begin{corollary}
If\label{ref28} all inverse branches of $T$ have positive determinant (i.e., $T$ is increasing on each interval), then $\Sigma_T$ equals either $\Gamma$ or its unique index-$2$ subgroup $\angles{R,SRS}$. It equals $\Gamma$ iff at least one of the matrices $A_a$ factors in $\Mcal$ as an $L$-$N$ product of odd length.
\end{corollary}
\begin{proof}
In the proof of Theorem~\ref{ref13} we discussed the five possible cases for $\Sigma_T$, cases~2 and 5 being void. Case~3 implies $RF\in\Sigma_T$, and Case~4 implies $R^2F\in\Sigma_T$; both are impossible here, since $\Sigma_T$ cannot contain elements of determinant $-1$. So we are left with Case~1; taking into account~(A), we conclude that $\Sigma_T$ extends~$\angles{R,SRS}$. Expressing $L,N$ in terms of $S,R$, one sees that $LL,LN,NL,NN\in\angles{R,SRS}$, while $L,N\notin\angles{R,SRS}$; our statements follow immediately.
\end{proof}

\section{The tail property}\label{ref15}

We briefly recall the definition of a finite-state transducer. Let an \newword{input alphabet}~$\Zcal$ and an 
\newword{output alphabet} $\Acal$ be given, both finite.
As usual, $\Zcal^*$ denotes the set of all finite words over $\Zcal$ (including the empty word), while $\Zcal^\omega$ is the set of all one-sided infinite sequences $\Zbf=Z_0Z_1Z_2\cdots$. A \newword{finite-state transducer} is a finite directed graph, whose edges are labelled by \newword{transition rules} of the form $Z|w$, with $Z\in\Zcal$ and $w\in\Acal^*$; we only consider \newword{deterministic} transducers, i.e., transducers such that, for each vertex and each $Z$, at most one edge labelled $Z|w$ leaves that vertex (in this context vertices are usually called \newword{states}).
Given the input $\Zbf$ and a vertex $v$ in a given set of initial states,
the transducer acts in the obvious way: it first checks if an edge labelled $Z_0|w_0$ starts from $v$. If so, it moves to the target vertex, checks if an edge labelled $Z_1|w_1$ starts from it, and goes on. The process stops and fails, producing no output, if a vertex is reached from which no appropriate edge starts; if this never happens the computation succeeds, yielding the output $\abf=w_0w_1w_2\cdots\in\Acal^\omega$.

For every slow algorithm $T$, the graph $\Gcal_T$ yields a canonical transducer as follows.
We take $\Zcal=\set{L,N}$ and $\Acal=\set{0,\ldots,n-1}$ as our alphabets;
it is expedient to equip $\Zcal^*$ and $\Zcal^\omega$ with the involution ${}'$ that exchanges $L$ with $N$ componentwise.
We first remove from $\Gcal_T$ all ``vertical'' $F$-edges, retaining however the handy notation $vX_0\cdots X_{r-1}$ for the vertex reached from $v$ by following the edges labelled
$\vect X0{{r-1}}\in\set{L,N,F}$, in this order.
For $Z\in\Zcal$, we relabel each $Z$-edge that does not end either at $1$ or at $2=1F$ by $Z|$. We now examine the edges terminating at $1$ or at $2$. Each
$A_a$ induces precisely two such edges; namely, as in Definition~\ref{ref17}, we write uniquely each $A_a$ as
$A_a=B_aF^{e(a)}=C_aZF^{e(a)}$, with $C_a\in\Mcal$. Then, by construction, $\Gcal_T$ contains a $Z$-edge from $C_a$ to either $1$ (if $e(a)=0$) or $2$ (if $e(a)=1$), as well as a $Z'$-edge from $C_aF$ to either $2$ (if $e(a)=0$) or $1$ (if $e(a)=1$).
We relabel these two edges by $Z|a$ and $Z'|a$, respectively; repeating this relabeling for each $A_a$, and taking $\set{1}$ as the set of initial states, leaves us with a transducer, again denoted by~$\Gcal_T$.
Since for each vertex of $\Gcal_T$ and each $Z$ precisely one edge labelled $Z|w$ starts from that vertex, this transducer is deterministic and succeeds at every input. As an example, Figure~\ref{fig14} shows ``one half'' (see~\S\ref{ref5}) of the transducer determined by the map $T$ in Example~\ref{ref26}.

For $\abf=a_0a_1\cdots\in\Acal^\omega$ and $\alpha\in[0,\infty]$, we write
$\sigma(\abf)=\alpha$ if $\abf$ is a $T$-symbolic orbit for $\alpha$, i.e., $T^t(\alpha)\in\Delta_{a_t}$ for every $t\ge0$.
Every irrational number has precisely one symbolic orbit, while every rational has at most two orbits. For the slow map $\set{L,N}$ of 
Example~\ref{ref16}, which has a distinguished status, we'll use $\pi:\Zcal^\omega\to [0,\infty]$ instead of~$\sigma$.
By the description in~\cite{series85},
writing $\pi(\Zbf)=\alpha$ is equivalent to saying that starting
from an arbitrary point $\beta$ in the imaginary axis $\Rbb\pp i$ of the hyperbolic upper-half plane and moving along the geodesic arc connecting $\beta$ to $\alpha$, the resulting cutting sequence of the Farey tessellation is $\Zbf$ (note that 
our $N,L$ are Series's $L,R$).
We abuse language by writing $\pi$ also for the map that associates to a word $W\in\Zcal^*$ the matrix in $\Gamma$ resulting from reading $W$ as an ordinary matrix product. This abuse is justified by the identity $\pi(W\Zbf)=\pi(W)*\pi(\Zbf)$, that is immediately proved by induction on the length of $W$.

\begin{lemma}
Let\label{ref18} $\pi(\Zbf)=\alpha$. Then the $T$-symbolic orbit of $\alpha$ is the output of $\Gcal_T$ at $\Zbf$.
\end{lemma}
\begin{proof}
The proof is straightforward, but we use it to introduce a formalism and an alternative description of $\Gcal_T$ that will be used later on.
We start at the root of~$\Gcal_T$ and follow the path determined by $\Zbf$, producing no output but taking note of our visits to $1$ and $2$ by inserting~$1$ and $2$ as indices along~$\Zbf$. We call the times $0=t_0<t_1<t_2<\cdots$ at which these visits take place the \newword{hitting times}. For example, 
let $\Zbf=NLLNNLNNNL\overline{LNL}$, whose $\pi$-image is $\alpha=(1335+\sqrt{3})/939$. Applying the transducer in Figure~\ref{fig7} we obtain
the annotated sequence
\begin{equation}\tag{$*$}
{}_1N_1LLN_2NLN_2NNL_1LNL_1LNL_1L\cdots.
\end{equation}
For $W\in\Zcal^*$, $\Ybf\in\Zcal^\omega$, $i,j\in\set{1,2}$, we call ${}_iW_j$ and ${}_i\Ybf$ a \newword{marked word} and a \newword{marked sequence}, respectively. We extend $\pi$ to marked words and sequences by defining
\begin{align*}
\pi({}_1W_1)&=\pi(W),&
\pi({}_1W_2)&=\pi(W)F,\\
\pi({}_2W_1)&=F\pi(W)=\pi(W')F,&
\pi({}_2W_2)&=\pi(W'),\\
\pi({}_1\Ybf)&=\pi(\Ybf),&
\pi({}_2\Ybf)&=F*\pi(\Ybf)=1/\pi(\Ybf).
\end{align*}
These definitions are concocted so that the identity $\pi({}_iW_j)*\pi({}_j\Ybf)=\pi({}_iW\Ybf)$ holds, as can be easily verified.

Now, each marked word ${}_iW_j$ appearing in a sequence such as~$(*)$ corresponds to a path from $i$ to $j$ making no intermediate visits to either $1$ or $2$; we call such a path \newword{primitive}. Also, by construction, upon completing a primitive path the transducer outputs the symbol $a$ determined by $\pi({}_iW_j)=A_a$.
We can then move along the annotated sequence and read the output $\abf=a_0a_1\cdots$ by computing~$\pi$ on the successive marked words. For example, the input
$\Zbf$ above yields the output
$\abf=4121\overline{2}$, and indeed $\pi({}_1N_1)=N=A_4$, $\pi({}_1LLN_2)=LLNF=A_1$, $\pi({}_2NLN_2)=LNL=A_2$, and so on.

The statement of the lemma follows now by induction on the set of hitting times.
Suppose that at the hitting time $t_k$ ($k\ge0$) the transducer produced the correct output $a_0\cdots a_{k-1}$, is in state $i$, and is to read the sequence ${}_iW_j\Ybf$ whose $\pi$-image is $T^k(\alpha)$. Let $\pi({}_iW_j)=A_a$; then $T^k(\alpha)=A_a*\pi({}_j\Ybf)\in\Delta_a$. It follows that the next symbol in the $T$-symbolic sequence of $\alpha$ is $a_k=a$, which agrees with the transducer output in its moving from $i$ to $j$ along $W$. Also, $T^{k+1}(\alpha)=A_a\m*T^k(\alpha)=\pi({}_j\Ybf)$, so that at the next hitting time $t_{k+1}$ the transducer is in state $j$ and is to read the sequence ${}_j\Ybf$, as required.
\end{proof}

The map $\sigma\m$ from the Baire space $\Rbb\pp\setminus\Qbb$ 
to the Cantor space $\Acal^\omega$ is injective and bicontinuous, hence
a homeomorphism from its domain to its image $\Acal^\sharp$.
Writing $\Sigma$ for $\Sigma_T$, on $\Acal^\sharp$ we have two equivalence relations:
\begin{itemize}
\item $\abf$ and $\bbf$ \newword{have the same tail}, written $\abf\equiv_t\bbf$, if there exist $h,k\ge0$ such that, for every $l\ge0$, $a_{h+l}=b_{k+l}$;
\item $\abf$ and $\bbf$ \newword{are $\Sigma$-equivalent}, written $\abf\equiv_\Sigma\bbf$, if $\sigma(\abf)$ and $\sigma(\bbf)$ are $\Sigma$-equivalent according to~\S\ref{ref1}(5).
\end{itemize}
As remarked in~\S\ref{ref1} and before Definition~\ref{ref19}, $\equiv_\Sigma$ is coarser than $\equiv_t$: each $\Sigma$-equivalence class is partitioned into tail-equivalence classes, and the Serret theorem holds for~$T$ precisely when these two relations agree.

\begin{definition}
In\label{ref22} the proof of Lemma~\ref{ref18} we noted that
the primitive paths in~$\Gcal_T$ correspond to marked words ${}_iW_j$ ($i,j\in\set{1,2}$). If $W=Z_0\cdots Z_{l-1}$, then the path has \newword{length} $l$ and contains the vertices $i,iZ_0,iZ_0Z_1,\ldots,iW=j$, which are all distinct, except for the possibility $i=j$. The $F$-twins of these vertices constitute another primitive path, namely ${}_{iF}W'_{jF}$, and $\pi({}_iW_j)=\pi({}_{iF}W'_{jF})=A_a$ for some $a$. Each $A_a$ corresponds as above to precisely these two primitive paths; let $D_a$ be the set of vertices along them, of cardinality $2l$. Let $\varphi\m\set{1}$ be the counterimage of $1$ under the opfibration $\varphi:\Gcal_T\to\Scal_T$. We define the \newword{defect} of $T$ to be the maximum, say $d$, of the cardinalities of $D_a\cap\varphi\m\set{1}$, as $a$ varies in $\set{0,\ldots,n-1}$.
\end{definition}

\begin{theorem}
Let\label{ref20} $\abf\in\Acal^\sharp$. Then the $\Sigma$-equivalence class of $\abf$ is partitioned in at most~$d$ 
tail-equivalence classes.
\end{theorem}

As an immediate corollary, the condition $\varphi\m\set{1}=\set{1}$ is sufficient for the validity of the Serret theorem. This condition holds, e.g., for the pythagorean map of Example~\ref{ref24}; by explicit computation one checks that $\Sigma$ has index $3$, while $\varphi\m\set{1}$ is trivial.
It is not a necessary condition, as witnessed by the map $T'$ of Example~\ref{ref26}; we'll provide a complete characterization in Corollary~\ref{ref27}.

\begin{example}
Let\label{ref21} $T$ be defined by
\begin{align*}
A_0&=LLL,&
A_1&=LLN,&
A_2&=LNLF,&
A_3&=LNN,\\
A_4&=NLL,&
A_5&=NLN,&
A_6&=NNL,&
A_7&=NNN.
\end{align*}
Computing the opfibration $\varphi:\Gcal_T\to\Scal(\Sigma)$ one sees that $\Scal(\Sigma)$ is trivial, so $\Sigma=\Pi$ and $\varphi\m\set{1}=(\text{all vertices of $\Gcal_T$})$.
Each primitive path in $\Gcal_T$ has length $3$, therefore $T$ has defect $6$.

Let $n_1,n_2,\ldots$ be any sequence in $\Zbb_{>0}$, let
\[
\Zbf=(LLNNLL)^{n_1}LLL(LLNNLL)^{n_2}LLL\cdots,
\]
and let $\alpha=\pi(\Zbf)$. Chasing paths along $\Gcal_T$, one checks that:
\begin{align*}
&\alpha &\text{has $T$-symbolic orbit}\qquad
&(14)^{n_1}0(14)^{n_2}0\cdots,\\
&F*\alpha &\text{has $T$-symbolic orbit}\qquad 
&(63)^{n_1}7(63)^{n_2}7\cdots,\\
&L\m*\alpha &\text{has $T$-symbolic orbit}\qquad 
&(30)^{n_1}0(30)^{n_2}0\cdots,\\
&FL\m*\alpha &\text{has $T$-symbolic orbit}\qquad 
&(47)^{n_1}7(47)^{n_2}7\cdots,\\
&L^{-2}*\alpha &\text{has $T$-symbolic orbit}\qquad
&(60)^{n_1}0(60)^{n_2}0\cdots,\\
&FL^{-2}*\alpha &\text{has $T$-symbolic orbit}\qquad 
&(17)^{n_1}7(17)^{n_2}7\cdots.
\end{align*}
Here the Serret theorem fails as badly as possible: there are $2^{\aleph_0}$ counterexamples, each of them breaking its $\Sigma$-equivalence class in the maximum available number of tail-equivalence classes.
\end{example}

The proof of Theorem~\ref{ref20} proceeds in three stages. We first note that, having started $\Gcal_T$ on input $\Zbf$ from the root, we'll arrive infinitely often
to a vertex $v$ such that either $\varphi(v)=1$ or $\varphi(vF)=1$.
Every time we hit such a $v$, we restart the transducer either from the root (if $\varphi(v)=1$) or from $2$ (if $\varphi(vF)=1$), feeding it with the remaining input. The resulting outputs are then the $T$-symbolic orbits of numbers $\beta$ $\Sigma$-equivalent to $\alpha$. In the second part of the proof we will show that every $\gamma$ $\Sigma$-equivalent to $\alpha$ is tail-equivalent to a $\beta$ of the form given by the first part. Finally, we will prove that the set of such $\beta$'s is partitioned in at most $d$ tail-equivalence classes.

\begin{proof}[Proof of Theorem~\ref{ref20}]
The first part is easy: let $\Zbf=W\Ybf$ be a factorization of $\Zbf$ such that $\pi({}_1W_i)\in\Sigma$ for a certain $i\in\set{1,2}$. Then $\beta=\pi({}_i\Ybf)\equiv_\Sigma\alpha$, and the proof of Lemma~\ref{ref18} shows that 
the $T$-symbolic orbit of $\beta$ is the output of $\Gcal_T$ when restarted in state $i$ and fed with $\Ybf$.

For the second part, let $\gamma=M*\alpha\in[0,\infty]$ for some $M\in\Sigma$, and let $W(t)=Z_0\cdots Z_{t-1}$ be the initial segment of $\Zbf$ of length $t$. Then $\pi\bigl(W(t)\bigr)*[0,\infty]$ is a unimodular interval that shrinks, as~$t$ increases, to $\alpha$, so that 
$M\pi\bigl(W(t)\bigr)*[0,\infty]$ is a unimodular interval shrinking to~$\gamma$.
Since $\gamma\in[0,\infty]$, we can choose $t$ (fixed from now on) so large that $[p/q,p'/q']=
M\pi\bigl(W(t)\bigr)*[0,\infty]$ is a subinterval of $[0,\infty]$; its extrema $p/q,p'/q'$ are the
$M\pi\bigl(W(t)\bigr)$-images of $0,\infty$, in this or the other order according to $M$ being of determinant $+1$ or $-1$. By \S\ref{ref1}(2), there exists a unique 
$U\in\Zcal^*$ such that $\pi(U)=\bbmatrix{p'}{p}{q'}{q}$. Therefore $M\pi\bigl(W(t)\bigr)\bbmatrix{0}{1}{1}{0}$ equals $\pi(U)F$ or $\pi(U)$, according to the positive or negative sign of $\det(M)$; in short, 
$M\pi\bigl(W(t)\bigr)=\pi(U)F^e$, with $e=0$ if $\det(M)=1$ and $e=1$ otherwise. As the sequence $\Xbf\in\Zcal^\omega$
such that $\pi(\Xbf)=\gamma$ is unique, we must have:
\begin{itemize}
\item if $e=0$, then $\Xbf=UZ_tZ_{t+1}\cdots$,
\item if $e=1$, then $\Xbf=UZ'_tZ'_{t+1}\cdots$.
\end{itemize}
Therefore there exists $q\ge t$ and $i\in\set{1,2}$ such that, writing $\Lcal$ for the monoid generated by $\vect A0{{n-1}}$, we have:
\begin{itemize}
\item if $e=0$, then $\pi\bigl({}_1(UZ_t\cdots Z_{q-1})_i\bigr)\in\Lcal$ and $\gamma\equiv_t\pi\bigl({}_i(Z_q Z_{q+1}\cdots)\bigr)$,
\item if $e=1$, then $\pi\bigl({}_1(UZ'_t\cdots Z'_{q-1})_i\bigr)\in\Lcal$ and $\gamma\equiv_t\pi\bigl({}_i(Z'_q Z'_{q+1}\cdots)\bigr)=
\pi\bigl({}_{iF}(Z_q Z_{q+1}\cdots)\bigr)$.
\end{itemize}

Now:
\begin{itemize}
\item if $e=0$, then
\begin{multline*}
\pi\bigl({}_1(Z_0\cdots Z_{q-1})_i\bigr)=\pi\bigl(W(t)\bigr)\,\pi\bigl({}_1(Z_t\cdots Z_{q-1})_i\bigr)\\
=M\m\pi(U)\,
\pi\bigl({}_1(Z_t\cdots Z_{q-1})_i\bigr)=M\m\,
\pi\bigl({}_1(UZ_t\cdots Z_{q-1})_i\bigr)
\end{multline*}
belongs to $\Sigma$,
\item if $e=1$, then
\begin{multline*}
\pi\bigl({}_1(Z_0\cdots Z_{q-1})_{iF}\bigr)=\pi\bigl(W(t)\bigr)\,\pi\bigl({}_1(Z_t\cdots Z_{q-1})_{iF}\bigr)\\
=M\m\pi(U)\,F\,
\pi\bigl({}_1(Z_t\cdots Z_{q-1})_{iF}\bigr)=M\m\,
\pi\bigl({}_1(UZ'_t\cdots Z'_{q-1})_i\bigr)
\end{multline*}
belongs to $\Sigma$ too.
\end{itemize}
We conclude that $\gamma$ is tail-equivalent to a $\beta$ of the form given in the first part of the proof, namely $\beta=\pi\bigl({}_i(Z_qZ_{q+1}\cdots)\bigr)$ if $e=0$, and $\beta=\pi\bigl({}_{iF}(Z_qZ_{q+1}\cdots)\bigr)$ if $e=1$.

We finally show that any set of $d+1$ $\beta$'s obtained as above must contain two tail-equivalent elements. Let us say that this set has been obtained by stopping $\Gcal_T$ on input $\Zbf$ at times $t_0<t_1<\cdots<t_d$, the transducer being in state $\vect v0d$, respectively. For each $0\le r\le d$,
let $i_r:=1$ if $\varphi(v_r)=1$ and $i_r:=2$ if $\varphi(v_rF)=1$.
We refer to the run of $\Gcal_T$ when restarted from state $i_r$ and fed with $Z_{t_r}Z_{t_r+1}\cdots$ as the \newword{$r$th run}.
Let $m_r$ be the limsup of the lengths of the primitive paths occurring during the $r$th run; without loss of generality $m=m_0\ge\vect m1d$.
Denoting the vertex $1Z_0\cdots Z_{q-1}$ by $v(q)$, we choose some
$q_0 > t_d$ such that
$\set{v(q_0),v(q_0+1),\ldots,v(q_0+m)}$ is a primitive path of length $m$ for the $0$th run,
while for every other run it is a path containing $1$ or $2$ at least once; this choice is possible due to our assumptions on $m$.
Let
\[
f:\set{q_0,\ldots,q_0+m-1}\times\set{1,2}\to(\text{vertices of $\Gcal_T$})
\]
be defined by
\[
f(q,j)=\begin{cases}
v(q),&\text{if $j=1$;}\\
v(q)F,&\text{if $j=2$.}
\end{cases}
\]
Then $f$ is injective, and its image is the set $D$ of all vertices and their $F$-twins in the primitive path of length $m$ for the $0$th run referred to above.
Note that the two final vertices $\set{v(q_0+m),v(q_0+m)F}$ are not missing from $D$, since they appear as the initial ones $\set{v(q_0),v(q_0)F}$.
Due to the assumptions on $q_0$, for each $0\le r\le d$ there exists
a pair $(q_r,j_r)\in\set{q_0,\ldots,q_0+m-1}\times\set{1,2}$
such that $j_r=i_rZ_{t_r}\cdots Z_{q_r-1}$.
We claim that $f(q_r,j_r)\in\varphi\m\set{1}$ for every $r$. Indeed, $\pi\bigl({}_1(Z_0\cdots Z_{t_r-1})_{i_r}\bigr)\in\Sigma$ and $\pi\bigl({}_{i_r}(Z_{t_r}\cdots Z_{q_r-1})_{j_r}\bigr)\in\Lcal$, so that
$\pi\bigl({}_1(Z_0\cdots Z_{q_r-1})_{j_r}\bigr)\in\Sigma$, as claimed.
Since the cardinality of $D\cap\varphi\m\set{1}$ is less than or equal to the defect $d$ of $\Gcal_T$, by the pigeonhole principle there must be $r\not=s\in\set{0,\ldots,d}$ such that $(q_r,j_r)=(q_s,j_s)$. This implies that the $r$th and the $s$th run yield tail-equivalent outputs.
\end{proof}

Theorem~\ref{ref20} (or, rather, its proof) yields a characterization of the $T$'s satisfying the tail property. It is expedient to introduce another transducer, $\Fcal_T$, whose set of states is $\varphi\m\set{1}=\set{1=P_1,P_2,\ldots,P_m}\subseteq\set{\text{states of }\Gcal_T}$. By the construction in Definition~\ref{ref17}, each $P_r$ is of the form $B$ or $BF$, where $B$ is a left factor of one of $\vect A0{{n-1}}$. For each $P_r$ and each $A_a$ we have a unique commutator relation $P_rA_a=A_{b_1}\cdots A_{b_q}P_s$ (the product $A_{b_1}\cdots A_{b_q}$ may be empty), and we add to $\Fcal_T$ an edge from $P_r$ to $P_s$ labelled with the transition rule $a|b_1\cdots b_q$.

\begin{lemma}
Let\label{ref29} $\alpha$ have $T$-symbolic orbit $\abf$. Then the $T$-symbolic orbit of $P_r*\alpha$ is the output of $\Fcal_T$, when starting from $P_r$ on input $\abf$.
\end{lemma}
\begin{proof}
Straightforward, using the facts that $\alpha=\lim_{t\to\infty}A_{a_0}A_{a_1}\cdots A_{a_{t-1}}*\infty$ and that $P_r*\text{--}\,$ is continuous.
\end{proof}

We save space by removing from $\Fcal_T$ the state $P_1$ and all edges entering it; since the edges leaving $P_1$ are loops labelled $a|a$, they are automatically removed. Call $\Fcal^*_T$ the resulting graph, which may be empty if $\varphi\m\set{1}=\set{1}$; this is the trivial case cited after the statement of Theorem~\ref{ref20}.

\begin{corollary}
Let\label{ref27} $T$, $\Fcal^*_T$ be as above. Then the Serret theorem holds for $T$ iff, for every input sequence $\abf\in\Acal^\omega$ and every state $P_r$ of $\Fcal^*_T$,
\begin{itemize}
\item either $\Fcal^*_T$ eventually stops;
\item or $\Fcal^*_T$ runs forever, producing an output $\bbf$ tail-equivalent to $\abf$.
\end{itemize}
\end{corollary}
\begin{proof}
The left-to-right direction is clear from Lemma~\ref{ref29}, noting that the stopping of $\Fcal^*_T$ amounts to $\Fcal_T$ entering state $P_1$. For the reverse direction, we assume that the Serret theorem fails for $T$ and construct $\alpha$ and $P_r$ such that $\alpha$ and $\beta:=P_r*\alpha$ have different $T$-tails (this implies $P_r\not= P_1$, so that $P_r$ is a state of $\Fcal^*_T$). By the proof of Theorem~\ref{ref20}, there exist an irrational number $\gamma=\pi(\Zbf)$, an integer $q>0$, and an exponent $e\in\set{0,1}$ such that:
\begin{itemize}
\item $\varphi(1Z_0\cdots Z_{q-1}F^e)=1$;
\item the tail of $\alpha:=\pi({}_{1F^e}Z_qZ_{q+1}\cdots)$ is different from the tail of $\gamma$.
\end{itemize}
Let $t\ge0$ be the greatest integer $<q$ such that $i:=1Z_0\cdots Z_{t-1}$ belongs to $\set{1,2}$. Then $\pi\bigl({}_i(Z_t\cdots Z_{q-1})_{1F^e}\bigr)\in\Sigma$, because
$\pi\bigl({}_1(Z_0\cdots Z_{q-1})_{1F^e}\bigr)\in\Sigma$
and
$\pi\bigl({}_1(Z_0\cdots Z_{t-1})_i\bigr)\in\Lcal$.
Therefore the state $iZ_t\cdots Z_{q-1}F^e$ belongs to $\varphi\m\set{1}$; say it is equal to $P_r$.
Now, $\beta:=P_r*\alpha=\pi({}_iZ_tZ_{t+1}\cdots)$ has the same tail as $\gamma$ (because $\pi\bigl({}_1(Z_0\cdots Z_{t-1})_i\bigr)\in\Lcal$), and thus has tail different from $\alpha$.
\end{proof}

It is easy to check that the graph $\Fcal^*_{T'}$ determined by the algorithm $T'$ of Example~\ref{ref26} contains just two states, $P_2=N$ and $P_3=NN$, each carrying a loop labelled $3|3$. The condition in Corollary~\ref{ref27} clearly holds, thus $T'$ has the tail property.

\section{Synchronizing words}\label{ref5}

Some algorithms fail the tail property in a very fragile way, in the sense that for Lebesgue almost every input $\alpha$ the tail- and $\Sigma$-equivalence classes of $\sigma\m(\alpha)$ coincide. This is surely the case when the graph $\Gcal_T$ is \newword{synchronizing}, i.e., admits a synchronizing word.
A \newword{synchronizing word} for a deterministic transducer is an input word~$W$ that resets the transducer, that is, leaves it in the same state, no matter which state we started with: $vW=uW$ for every two states $v,u$~\cite{volkov08}.

\begin{theorem}
Let\label{ref23} $\Gcal^1_T$ be the connected component of $1$ in $\Gcal_T$; it equals all of $\Gcal_T$ iff $\Sigma_T\not\le\Gamma$. Assume that 
$\Gcal^1_T$ is synchronizing. Then the set of $\alpha\in\Rbb\pp\setminus\Qbb$ whose $\Sigma_T$-equivalence class is partitioned in more than one tail-equivalence classes has Lebesgue measure $0$.
\end{theorem}
\begin{proof}
Let $R:\Rbb\pp\setminus\Qbb\to\Rbb\pp\setminus\Qbb$ be the slow map of Example~\ref{ref16}, explicitly given by
\[
R(x)=\begin{cases}
x/(1-x),&\text{if $x\in\Delta_0=[0,1]$;}\\
x-1,&\text{if $x\in\Delta_1=[1,\infty]$.}
\end{cases}
\]
Then $R$ preserves the $\sigma$-finite, infinite measure $\ud\mu=\ud x/x$, and is conservative and ergodic w.r.t.~it (see, e.g., \cite{isola11}).
Let $W=Z_0\cdots Z_{s-1}$ be a synchronizing word for~$\Gcal^1_T$, let $a(t)$ be $0$ or $1$ according whether $Z_t$ is $L$ or $N$, and let $B=\bigcap_{t=0}^{s-1}R^{-t}\Delta_{a(t)}$. The conservativity and ergodicity of $R$ easily imply (this is really a version of the Poincar\'e recurrence theorem) that $\mu$-all points enter any set of positive $\mu$-measure infinitely often. Since:
\begin{itemize}
\item[(i)] $\mu(B)>0$,
\item[(ii)] $\mu$ and the Lebesgue measure have the same nullsets,
\item[(iii)] $\pi:(\Zcal^\omega,\text{shift})\to(\Rbb\pp\setminus\Qbb,R)$ is a measurable conjugacy,
\end{itemize}
we conclude that for Lebesgue-all $\alpha=\pi(\Zbf)$ the input $\Zbf$ to $\Gcal_T$ contains $W$ infinitely often.
Our statement then follows from the proof of Theorem~\ref{ref20}.
\end{proof}

We refer to~\cite[Corollary~4.3]{katokugarcovici12} for a result in the same vein, albeit stated in a different context and proved with different means.

Theorem~\ref{ref23} applies, e.g., to the map $T$ in Example~\ref{ref26}, whose $\Gcal^1_T$ is in Figure~\ref{fig14}.
\begin{figure}[h!]
\begin{tikzpicture}[scale=1.2]
\node (1) at (0,0)  []  {$1$};
\node (3)  at (2,0) [vertex1]  {};
\node (5)  at (4,0) [vertex1]  {};
\path
(1) edge[Nedge,out=150,in=-150,loop] node[left] {$L|0$} (1)
(1) edge[Nedge,out=-30,in=-150] node[below,inner sep=1pt] {$N|$} (3)
(3) edge[Nedge,out=150,in=30] node[above,inner sep=1pt] {$N|3$} (1)
(3) edge[Nedge] node[above,inner sep=1pt] {$L|$} (5)
(5) edge[Nedge,out=-120,in=-60] node[below,inner sep=1pt] {$N|2$} (1)
(5) edge[Nedge,out=120,in=60] node[above,inner sep=1pt] {$L|1$} (1);
\end{tikzpicture}
\caption{}
\label{fig14}
\end{figure}
The word $LL$ is synchronizing for this graph, resetting it
to state~$1$. Note that the exceptional points 
$\alpha=\sqrt{3}$, $\beta=\sqrt{3}+1$ 
are the $\pi$-images of the sequences $\overline{NLN}$, $\overline{NNL}$, which avoid $LL$, as well as any other synchronizing word.

Again, Example~\ref{ref21} provides a much more robust counterexample to the Serret theorem. Indeed, explicit computation (which is not trivial, since $\Gcal_T$ has $14$ vertices, so that the synchronizability criterion in~\cite[Proposition~1]{volkov08} involves a graph with $105$ vertices) shows that the associated transducer does not admit any synchronizing word.


\begin{thebibliography}{10}

\bibitem{arnouxschmidt13}
P.~Arnoux and T.~A. Schmidt.
\newblock Cross sections for geodesic flows and {$\alpha$}-continued fractions.
\newblock {\em Nonlinearity}, 26(3):711--726, 2013.

\bibitem{baladivallee05}
V.~Baladi and B.~Vall{\'e}e.
\newblock Euclidean algorithms are {G}aussian.
\newblock {\em J. Number Theory}, 110(2):331--386, 2005.

\bibitem{bocalinden17}
F.~P. Boca and C.~Linden.
\newblock On {M}inkowski type question mark functions associated with even or
  odd continued fractions.
\newblock \texttt{https://arxiv.org/abs/1705.01238}, 2017.

\bibitem{boldivigna02}
P.~Boldi and S.~Vigna.
\newblock Fibrations of graphs.
\newblock {\em Discrete Math.}, 243(1-3):21--66, 2002.

\bibitem{borwein_et_al14}
J.~Borwein, A.~van~der Poorten, J.~Shallit, and W.~Zudilin.
\newblock {\em Neverending fractions}, volume~23 of {\em Australian
  Mathematical Society Lecture Series}.
\newblock Cambridge University Press, 2014.

\bibitem{CornfeldFomSi82}
I.~P. Cornfeld, S.~V. Fomin, and Ya.~G. Sina\u{\i}.
\newblock {\em Ergodic theory}, volume 245 of {\em Grundlehren der
  Mathematischen Wissenschaften}.
\newblock Springer, 1982.

\bibitem{einsiedlerward}
M.~Einsiedler and T.~Ward.
\newblock {\em Ergodic theory with a view towards number theory}, volume 259 of
  {\em Graduate Texts in Mathematics}.
\newblock Springer, 2011.

\bibitem{grabinerlagarias01}
D.~J. Grabiner and J.~C. Lagarias.
\newblock Cutting sequences for geodesic flow on the modular surface and
  continued fractions.
\newblock {\em Monatsh. Math.}, 133(4):295--339, 2001.

\bibitem{hardywri85}
G.~H. Hardy and E.~M. Wright.
\newblock {\em An introduction to the theory of numbers}.
\newblock Oxford University Press, 5th edition, 1985.

\bibitem{heersink16}
B.~Heersink.
\newblock An effective estimate for the {L}ebesgue measure of preimages of
  iterates of the {F}arey map.
\newblock {\em Adv. Math.}, 291:621--634, 2016.

\bibitem{iommi10}
G.~Iommi.
\newblock Multifractal analysis of the {L}yapunov exponent for the backward
  continued fraction map.
\newblock {\em Ergodic Theory Dynam. Systems}, 30(1):211--232, 2010.

\bibitem{isola11}
S.~Isola.
\newblock From infinite ergodic theory to number theory (and possibly back).
\newblock {\em Chaos Solitons Fractals}, 44(7):467--479, 2011.

\bibitem{katokugarcovici12}
S.~Katok and I.~Ugarcovici.
\newblock Applications of {$(a,b)$}-continued fraction transformations.
\newblock {\em Ergodic Theory Dynam. Systems}, 32(2):755--777, 2012.

\bibitem{leveque77}
W.~J. LeVeque.
\newblock {\em Fundamentals of number theory}.
\newblock Dover, 1996.
\newblock Reprint of the 1977 original.

\bibitem{liardetstambul98}
P.~Liardet and P.~Stambul.
\newblock Algebraic computations with continued fractions.
\newblock {\em J. Number Theory}, 73(1):92--121, 1998.

\bibitem{LindMar95}
D.~Lind and B.~Marcus.
\newblock {\em An introduction to symbolic dynamics and coding}.
\newblock Cambridge University Press, 1995.

\bibitem{nakada81}
H.~Nakada.
\newblock Metrical theory for a class of continued fraction transformations and
  their natural extensions.
\newblock {\em Tokyo J. Math.}, 4(2):399--426, 1981.

\bibitem{pinner01}
C.~G. Pinner.
\newblock More on inhomogeneous {D}iophantine approximation.
\newblock {\em J. Th\'eor. Nombres Bordeaux}, 13(2):539--557, 2001.

\bibitem{raney73}
G.~N. Raney.
\newblock On continued fractions and finite automata.
\newblock {\em Math. Ann.}, 206:265--283, 1973.

\bibitem{romik08}
D.~Romik.
\newblock The dynamics of {P}ythagorean triples.
\newblock {\em Trans. Amer. Math. Soc.}, 360(11):6045--6064, 2008.

\bibitem{series85}
C.~Series.
\newblock The modular surface and continued fractions.
\newblock {\em J. London Math. Soc. (2)}, 31(1):69--80, 1985.

\bibitem{serret66}
J.~A. Serret.
\newblock {\em Cours d'alg\`ebre sup\'erieure}.
\newblock Gauthier-Villars, 3rd edition, 1866.

\bibitem{volkov08}
M.~V. Volkov.
\newblock Syncronizing automata and the {{\v C}}ern{\'y} conjecture.
\newblock In Martín-Vide C., Otto F., and Fernau H., editors, {\em Language
  and Automata Theory and Applications}, volume 5196 of {\em Lecture Notes in
  Computer Science}, pages 11--27. Springer, 2008.

\bibitem{zagier75}
D.~Zagier.
\newblock Nombres de classes et fractions continues.
\newblock {\em Ast\'erisque}, 24-25:81--97, 1975.

\end{thebibliography}

\end{document}